\documentclass[12pt,reqno]{amsart}
\usepackage[top=1in, bottom=1in, left=1in, right=1in]{geometry}
\usepackage{times, amsthm, amssymb, amsmath, amsfonts, amsthm,mathrsfs,amsrefs,mathtools,bm}
\usepackage{hyperref}
\usepackage[usenames,dvipsnames]{color}
\usepackage{enumitem}
\usepackage{tikz}
\usepackage{tikz-cd}
\usepackage{tikz-3dplot}
\usepackage{subcaption}
\usepackage{caption}
\usepackage{xcolor}
\usepackage[capitalise]{cleveref}
\usepackage{adjustbox}
\usepackage[T1]{fontenc}
\usepackage{booktabs}
\usepackage{tabularx}
\newtheorem{theorem}{Theorem}[section]
\newtheorem{lemma}[theorem]{Lemma}
\newtheorem{proposition}[theorem]{Proposition}
\newtheorem{corollary}[theorem]{Corollary}

\theoremstyle{definition}
\newtheorem{definition}[theorem]{Definition}
\newtheorem{example}[theorem]{Example}
\newtheorem{remark}[theorem]{Remark}
\newcommand{\<}{\langle}
\renewcommand{\>}{\rangle}
\newcommand{\minus}{\smallsetminus}

\newcommand{\grain}{\mathsf{G}}
\newcommand{\N}{\mathbb{N}} 
\newcommand{\Z}{\mathbb{Z}} 
\newcommand{\Q}{\mathbb{Q}}
\newcommand{\R}{\mathbb{R}} 
\newcommand{\field}{\Bbbk} 
\newcommand{\affinesemigp}{Q} 
\newcommand{\genset}{A}

\newcommand{\realize}{\R_{\geq0}} 
\newcommand{\dimd}{d} 
\newcommand{\diml}{\ell}
 
\newcommand{\facF}{F} 
\newcommand{\facG}{G} 
\newcommand{\facH}{H}

\newcommand{\virtual}[1]{\widehat{#1}} 

\DeclareMathOperator{\relint}{RelInt}

\newcommand{\polyhedron}{P} 
\newcommand{\polytope}{\mathcal{P}} 
\newcommand{\crosssection}{K}

\newcommand{\midI}{I} 
\newcommand{\idI}{T}

\newcommand{\maxid}{\mathfrak{m}} 
\newcommand{\module}{M} 
\newcommand{\localcoho}[3]{H_{#1}^{#2}(#3)} 
\newcommand{\setS}{S} 
\newcommand{\setT}{T}

\newcommand{\smalln}{n} 
\newcommand{\smallm}{m}
\newcommand{\smallk}{k} 
\newcommand{\smalli}{i}

\newcommand{\smalla}{a}

\newcommand{\hplane}{\mathscr{H}} 

\newcommand{\veca}{\mathbf{a}}
\newcommand{\vecb}{\mathbf{b}} 
\newcommand{\vecc}{\mathbf{c}}

\newcommand{\vecf}{\mathbf{f}} 
\newcommand{\vecg}{\mathbf{g}}

\newcommand{\vect}{\mathbf{t}}

\newcommand{\vecx}{\mathbf{x}}

\newcommand{\varx}{x} 
\newcommand{\vary}{y} 
\newcommand{\varz}{z} 
\newcommand{\vart}{t} 
\newcommand{\ishida}{L} 
\newcommand{\rchaincpx}[1]{\tilde{\mathcal{C}}(#1)}

\newcommand{\facelattice}{\mathcal{F}} 
\newcommand{\incidence}{\epsilon} 
\newcommand{\hilbt}{\mathbf{t}} 
\newcommand{\hilb}[2]{\operatorname{Hilb}(#1,#2)} 
 
\newcommand{\tensor}[1]{
  \mathbin{\mathop{\otimes}\displaylimits_{#1}}
}

\newcommand{\holes}{\mathcal{H}} 
\newcommand{\arngmt}{\mathcal{A}}
\newcommand{\regionr}[1]{\mathfrak{r}(#1)} 
\newcommand{\regionR}[1]{\mathfrak{R}(#1)} 
\newcommand{\regr}{\mathfrak{r}} 
\newcommand{\regR}{\mathfrak{R}}

\newcommand{\catname}[1]{{\normalfont\textbf{#1}}}
\DeclareMathOperator{\Hom}{Hom} 
 
\DeclareMathOperator{\Span}{Span} 
\DeclareMathOperator{\nat}{nat} 
\DeclareMathOperator{\degp}{deg.p} 
\DeclareMathOperator{\grainset}{\mathcal{G}} 
\DeclareMathOperator{\degs}{\bigcup\deg} 
\DeclareMathOperator{\degpo}{\overline{deg.p}} 
\DeclareMathOperator{\voidm}{Holes}

\DeclareMathOperator{\rest}{res} 

\DeclareMathOperator{\face}{face}

\usepackage{scalerel}
\usepackage{graphicx}

\colorlet{ignored}{gray!50}
\colorlet{idealcolor}{black}
\colorlet{stdmcolor}{blue!60}
\colorlet{idealregioncolor}{gray!20}

\makeatletter
\@namedef{subjclassname@2020}{
  \textup{2020} Mathematics Subject Classification}
\makeatother

\begin{document}
\title{Graded local cohomology of modules over semigroup rings}

\author[Laura Felicia Matusevich]{Laura Felicia Matusevich}
\address[Laura Felicia Matusevich]{Department of Mathematics \\
Texas A\&M University \\ College Station, TX 77843.}
\email[Laura Felicia Matusevich]{matusevich@tamu.edu}
\author[Byeongsu Yu]{Byeongsu Yu}
\address[Byeongsu Yu]{Department of Mathematics \\
Texas A\&M University \\ College Station, TX 77843.}
\email[Byeongsu Yu]{byeongsu.yu@tamu.edu}

\date{}

\begin{abstract}
We give a combinatorial description of local cohomology modules of a
graded module over a semigroup ring, with support
at the graded maximal ideal. This combinatorial framework yields
Hochster-type formulas for the Hilbert series of such local cohomology
modules in terms of the homology of finitely many polyhedral cell
complexes. A Cohen--Macaulay criterion immediately follows. We also
provide an alternative proof of a result of~\cite{MR857437}
characterizing Cohen--Macaulay affine semigroup rings. 
\end{abstract}

  

\maketitle

\section{Introduction}
\label{sec:intro}

An affine semigroup ring is a subalgebra of the Laurent polynomial
ring (over a field) that is finitely generated by monomials. Such
rings are the coordinate rings of affine toric varieties, and have
received much attention for this reason.

Local cohomology is a fundamental notion in homological commutative
algebra. In this article, we make a detailed study of the graded
structure of local cohomology (supported at the graded maximal ideal) 
for modules over affine semigroup rings. In this case, one computes
local cohomology using the \emph{Ishida complex}, which takes into
account the combinatorics of the underlying affine
semigroup.

For a given graded module (over an affine semigroup ring), we need to study all localizations by
monomials. To do this, we generalize a tool introduced in the context of monomial
ideals in polynomial rings, namely the standard pairs of Sturmfels,
Trung and Vogel~\cites{STV95,STDPAIR}. In the general
context of graded modules, we use the name \emph{degree pairs}. These pairs 
organize the supporting degrees of a module according to the faces of the
underlying semigroup. 
Since localization is also controlled by 
faces, the degree pairs naturally control the information associated
to localization.

Putting all degree pairs of all localizations together, and suitably
topologizing, we can partition the relevant supporting degrees that
appear in the Ishida complex. The parts consist of lattice points in
carefully constructed polyhedra. Understanding these polyhedra (and
associated polyhedral cell complexes) we can classify the graded
pieces of the Ishida complex.

The key point of these constructions is that there are only \emph{finitely
many} polyhedra involved. In particular, we may follow foundational
ideas from Stanley--Reisner rings, and write Hilbert series for local
cohomology as finite sums involving lattice point generating
functions. When the semigroup is pointed, we write these Hilbert
series as finite sums of rational functions, thus obtaining
Hochster-type formulas in this case. Cohen--Macaulayness
criteria directly follow. 

Our tools also provide alternative proofs of celebrated results for
affine semigroup rings, namely Hochster's theorem that normal affine
semigroup rings are Cohen--Macaulay, and the Cohen--Macaulayness
criterion for affine semigroup rings by Trung and Hoa~\cite{MR857437}.

\subsection*{Outline}
Polyhedral geometry, affine semigroups, hyperplane arrangements, and the Ishida complex are discussed in~\cref{sec:preliminary}.~\cref{sec:stdpairs} introduces degree pairs and develops an injective map from the set of overlap classes of degree pairs of a localization to that of an original module.~\cref{sec:stdpair_top} introduces the concept of \emph{degree space}, which is composed of all degrees of nonzero elements in any localizations of an affine semigroup. The degree space is endowed with a special topology called \emph{degree pair topology}. ~\cref{sec:Hochster} derives a Hochster-type formula for the local cohomology from the degree pair topology and proposes combinatorial Cohen--Macaulay criteria for quotients of affine semigroup rings by monomial ideals. Finally,~\cref{sec:Cohen--Macaulay} proves the combinatorial Cohen–Macaulay condition for affine semigroup rings in a different way.

\subsection*{Acknowledgments}
We are grateful to Aida Maraj, Aleksandra Sobieska, Alexander Yong, Bernd Siebert, Catherine Yan, Christine Berkesch, Christopher Eur, Erika Ordog, Ezra Miller, Frank Sotille, Galen Dorpalen-Barry, Heather Harrington, Jaeho Shin, Jennifer Kenkel, Jonathan Monta{\~n}o, Joseph Gubeladze, Kenny Easwaran, Mahrud Sayrafi, Mateusz Michalek, Melvin Hochster, Patricia Klein, Sarah Witherspoon, Semin Yoo, Serkan Ho{\c s}ten, Yupeng Li for inspiring conversations we had while working on this project. 

\subsection*{Notation}
We adopt the convention that $\N=\{ 0,1,2, \cdots\}$ is the set of nonnegative integers, $\field$ is an arbitrary infinite field, and $\R$ is a field of real numbers. Throughout this article, $a,b,c,\cdots$ refer to integers, while $\veca,\vecb,\vecc,\cdots$ refer to integer vectors. For any set $\setS$, $\setS^{c}$ is the complement of the set. Given a Laurent polynomial ring $\field[t_{1}^{\pm1},t_{2}^{\pm1},\cdots, t_{n}^{\pm1}]$, let $\hilbt^{\veca}:= \vart_{1}^{\veca_{1}}\vart_{2}^{\veca_{2}} \cdots \vart_{n}^{\veca_{n}}$. When the context is obvious, $\varx, \vary$ and $\varz$ denote, respectively, $\vart_{1},\vart_{2},$ and $\vart_{3}$.

\section{Preliminaries}
\label{sec:preliminary}

\subsection{Polyhedral geometry}
\label{subsec:polytope}

A \emph{polyhedron} is an intersection of finitely many
closed halfspaces in $\R^{\dimd}$. A \emph{convex polytope} is a
bounded polyhedron. Equivalently, a convex polytope is the convex hull
of a finite set of points in $\R^{\dimd}$~\cite{Ziegler95}*{Theorem
  1.1}.
A \emph{polyhedral cone} is a polyhedron $\polytope$ which is also a
cone, meaning that $\lambda \vecx \in \polytope$ for all nonnegative
real $\lambda$ and all $\vecx \in \polytope$.

A nonempty subset $\facF$ of a polyhedron $\polytope$ in $\R^\dimd$ is called a
\emph{face} of $\polytope$ if there exists $\vecc\in\R^{\dimd}$ such
that $F$ is the set where the dot product with $\vecc$ over
$\polytope$ achieves its maximum value. 
In other words, 
$\facF = \face_{\vecc}(\polytope):=\{ \veca \in \polytope\mid \vecc \cdot \veca
\geq \< \vecc,\vecx \> \text{ for all } \vecx \in \polytope\}$. 
The faces of a polyhedron are also polyhedra. A face of $\polytope$ is
\emph{proper} if it is a proper subset of $\polytope$. By convention,
$\varnothing$ is a face of every polyhedron.

The \emph{affine hull} of a finite set \
$\veca_{1},\veca_{2},\cdots, \veca_{m} \in \R^\dimd$ is the set of all real linear
combinations $\sum_{i=1}^{m}\lambda_{i}\veca_{i}$ 
for which $\sum_{i=1}^{m}\lambda_{i}=1$. The \emph{relative interior} $\relint(\polytope)$ of a
polyhedron $\polytope$ is the interior of $\polytope$ with respect to its
affine hull. If $\facF$ is a face of $\polytope$ then $\relint(\facF)$
can also be described as the set of all points in $\facF$ that do not
lie in any other proper face of $\polytope$.
The dimension of a nonempty polyhedron is defined to be the dimension of its
affine hull; we set $\dim(\varnothing) = -1$.
A polyhedron is \emph{pointed} if
it has a unique zero-dimensional face.  It is easy to see that a
pointed polyhedron whose unique zero-dimensional face is the origin is
a (pointed) polyhedral cone.

The collection  $\facelattice(\polytope)$  of all faces in $\polytope$ is called
the \emph{face lattice of $\polytope$}. It is a lattice with respect to
the partial order given by inclusion.
Two polyhedra are \emph{combinatorially
  equivalent} if their face lattices are order-isomorphic.
If $\polytope$ is a pointed
polyhedral cone, then there exists a convex
polytope whose face lattice is order-isomorphic to
$\facelattice(\polytope) \smallsetminus \{ \varnothing
\}$~\cite{Ziegler95}*{Proposition 1.12, Exercise 2.19}.

A \emph{polyhedral complex} $\Delta$ is a collection of polyhedra
satisfying
\begin{itemize}
  \item for $\facF \in \Delta$ and $\facG \in \facelattice(\facF)$, we
    have $\facG \in \Delta$, and
  \item for $\facF, \facG \in \Delta$, $\facF\cap\facG$ is a common
      face of both $\facF$ and $\facG$.
\end{itemize}
    The face lattice of a polyhedron is an
example of a polyhedral complex. Any simplicial complex on $n$
vertices can be realized as a polyhedral subcomplex of the simplex
obtained by taking the convex hull of the standard basis vectors in $\R^n$.

\subsection{Affine semigroups}
\label{subsec:affinesemigp}

An \emph{affine semigroup} is a finitely generated submonoid of
$\Z^\dimd$. Throughout this article, we denote $\genset=\{ \veca_{1},\dots,
\veca_{n}\} \subset \Z^{\dimd} \smallsetminus \{ 0\}$ and work with
the affine monoid $\affinesemigp = \N\genset$ consisting of all nonnegative linear
combinations of the elements of $\genset$. Where it causes no
confusion, we also denote by $\genset$ the $d \times n$ integer matrix
with columns $\veca_{1},\dots, \veca_{n}$. The set
$\realize{\affinesemigp} = \realize{\genset}$ of nonnegative real
combinations of elements of $\affinesemigp$ (or $\genset$) is a
polyhedral cone, called the \emph{underlying cone of $\affinesemigp$}.
The \emph{dimension} of an affine semigroup $\affinesemigp$
is defined to be dimension of its underlying cone.

A subset
$\idI$ of an affine semigroup $\affinesemigp$ is called an \emph{ideal} if
$\affinesemigp+\idI \subseteq \idI$.
For any subset $\setS$ of
$\affinesemigp$, the ideal $\< \setS\>$ generated by $\setS$ is the
smallest ideal in $\affinesemigp$ that contains $\setS$. An ideal
$\idI$ is \emph{prime} if for any two elements $\veca,\vecb \in
\affinesemigp$, $\veca+\vecb \in \idI$ implies $\veca \in \idI$ or
$\vecb \in \idI$. A subset of $Q$ is called a \emph{face} of $Q$ if
its complement is a prime ideal; the collection of faces of
$\affinesemigp$ is denoted by $\facelattice(\affinesemigp)$.
It can be shown~\cite{CCA}*{Lemma~7.12} that there is a one to one
correspondence between $\facelattice(\realize{\affinesemigp})$ and
$\facelattice(\affinesemigp)$, given by intersecting the faces of
$\realize{\affinesemigp}$ with $\affinesemigp$.
The \emph{relative interior} of a face of $\affinesemigp$ is defined
to be the intersection of the relative interior of the corresponding
face of $\realize{\affinesemigp}$ with $\affinesemigp$.
If the context is clear, we may abuse notation and refer to a subset $\facF
\subset \genset$ as a face of $\N\genset$, to indicate that $\N\facF$ is a
face of $\N\genset$.

The set $\holes(\affinesemigp) := \left(\Z\affinesemigp \cap
  \realize\affinesemigp\right) \smallsetminus \affinesemigp$ is called
the set of \emph{holes} of $\affinesemigp$. Here $\Z\affinesemigp$
denotes the set of integer combinations of the elements of $\affinesemigp$.
If an affine semigroup contains no holes, it is said
to be \emph{normal}.

An \emph{affine semigroup ring}
$\field[\affinesemigp]=\field[\vart^{\veca_{1}},\cdots,
\vart^{\veca_{n}}]$ is a subring of the Laurent polynomial ring
$\field[\hilbt^{\pm}] =\field[\vart_{1}^{\pm1},\cdots,
\vart_{d}^{\pm1}]$. There is a natural bijection between the elements
of an affine semigroup $\affinesemigp$ and the monomials of the
corresponding affine semigroup ring $\field[\affinesemigp]$. This
establishes a one to one correspondence between monomial ideals of
$\field[\affinesemigp]$ and ideals of $\affinesemigp$. If $\idI$ is an
ideal of $\affinesemigp$, we denote the corresponding monomial ideal of
$\field[\affinesemigp]$ by $\midI$; more precisely, $\midI= \<
\varx^{\veca}\mid \veca \in \idI \>.$ If $F$ is a face of $Q$, the
ideal $P_F := \< t^\veca \mid \veca \notin F\>$ is a corresponding to
the complement of $F$ is a prime monomial ideal of
$\field[\affinesemigp]$; all prime monomial ideals of
$\field[\affinesemigp]$ arise in this way.

A set $\setS \subseteq \affinesemigp$ is called \emph{additively
  closed} if it contains $0$ and is closed under addition. The
\emph{localization} $\affinesemigp-\N\setS$ of $\affinesemigp$ by an
additively closed set $\setS$ is defined as
$\affinesemigp-\N\setS:=\affinesemigp + \Z \setS$. The localization of
$\affinesemigp$ by $\setS$ is equal to the localization of
$\affinesemigp$ by the minimal additively closed set containing
$\setS$ whose complement is a prime ideal~\cite{CHWW15}*{Lemma~1.1}.  

\begin{lemma}[\cite{CHWW15}*{Lemma 1.1}]
\label{lem:localization_proposition}
Let $\setS \subseteq \field[\affinesemigp]$ be a set of monomials that
is multiplicatively closed and let $\N\facF$ be the minimal face of $\affinesemigp$ containing $\{ \veca \in \affinesemigp\mid x^{\veca} \in \setS \}$. Then,
\[
\setS^{-1}\field[\affinesemigp] \cong \field[\affinesemigp - \N\facF].
\]
\end{lemma}

If $T \subset Q$ is an ideal corresponding to the monomial ideal
$\midI$ in $\field[\affinesemigp]$, and $F$ is a face of $Q$, the
localization of $I$ at the prime ideal $P_F$, denoted $\midI_F$
corresponds to the ideal $\idI_{\facF}:= \idI -\N\facF$ of the semigroup $\affinesemigp-\N\facF$.

A graded module $\module = \oplus_{\veca \in \Z^d} \module_\veca$ is
\emph{finely graded} if $\dim_{\field}\module_{\veca} \leq 1$ for all
degrees $\veca \in \Z^d$. For any face $\facF \in
\facelattice(\affinesemigp)$, $\field[\affinesemigp-\N\facF]$ is
\emph{finely} $\Z^{d}$-graded as follows. 

\begin{lemma} 
\label{lem:Zd-graded_property_of_affine_semigroup_ring}
$\dim_{\field}( \field[\affinesemigp-\N\facF])_{\veca}=1$ if $\veca \in \affinesemigp - \N\facF$. Otherwise, $\dim_{\field}( \field[\affinesemigp-\N\facF])_{\veca}=0$.
\end{lemma}

\begin{proof}
As $\field[\affinesemigp - \N\facF] \subseteq
\field[\Z\affinesemigp]$, it suffices to show that $\dim_{\field}(
\field[\Z\affinesemigp])_{\veca} \leq 1$. If $\veca \not\in
\Z\affinesemigp$, then $\field[\Z\affinesemigp]_{\veca} = \{ 0 \}$,
otherwise $\field[\Z\affinesemigp]_{\veca} = \Span_\field(t^\veca)$.
\end{proof}

An affine semigroup $\affinesemigp$ is \emph{pointed} if its corresponding cone
$\realize{\affinesemigp}$ is a pointed polyhedron.

\begin{example}
\label{ex:affine_semigp}
\
\begin{enumerate}[leftmargin=*]
\item \label{enum:frob_affine_semigp} (Monomial curves)
Let $\affinesemigp=\N\genset$ with $\genset=\left[\begin{smallmatrix}
    1 & 1 & \cdots & 1 \\ 0 & a_{1}  & \cdots &
    a_{\smalln-1}\end{smallmatrix}\right]$ such that $0<a_{1}< \cdots
< a_{\smalln-1}$ are relatively prime integers. If $0,
a_1,a_2,\dots,a_{n-1}$ are consecutive integers, then
$\field[\affinesemigp]$ is the coordinate ring of a rational normal
curve; otherwise, $\field[\affinesemigp]$ is not normal. 

Figure~\ref{fig:frob_coin} illustrates the example where $a_1=1, a_2=3,
a_3=4$. Elements of the semigroup $Q$ are represented by filled dots. Since
$\left[\begin{smallmatrix} 1 \\ 2 \end{smallmatrix}\right]$ is a hole
of $Q$ (in fact, it is the only hole of $Q$), it is depicted as an empty circle.
Let $\idI=\left\<\left[\begin{smallmatrix} 1 \\
      1 \end{smallmatrix}\right]\right\>$. The elements of $\idI$ are
colored black, while the elements in $Q$ but not in $\idI$ are colored
blue. This includes $\left[\begin{smallmatrix} 2 \\
    2 \end{smallmatrix}\right]$, which is not in $\idI$ because 
$\left[\begin{smallmatrix} 1 \\ 2 \end{smallmatrix}\right]$ is a hole
of $Q$.
\item 
\label{enum:3d_affine_semigp} (Segre embedding of $\mathbb{P}^{1} \times \mathbb{P}^{1}$)
Let $\affinesemigp=\N\genset$ with $A = \left[\begin{smallmatrix}0 & 1
    & 1 & 0 \\ 0 & 0 & 1 & 1 \\ 1 & 1 & 1 &
    1 \end{smallmatrix}\right]$. The affine semigroup ring
$\field[\affinesemigp]$ is isomorphic to $\field[z,xz,yz,xyz] \cong
\field[a,b,c,d]/\< ac-bd \>$. The exponent vectors of monomials in the
ideal $\midI=\< x^2 z^2, x^2yz^2, x^2y^3z^3,$ $x^3y^3z^3 \> \subset
\field[\affinesemigp]$ are depicted by black dots in~\cref{fig:3d}. 
\end{enumerate}
\end{example}

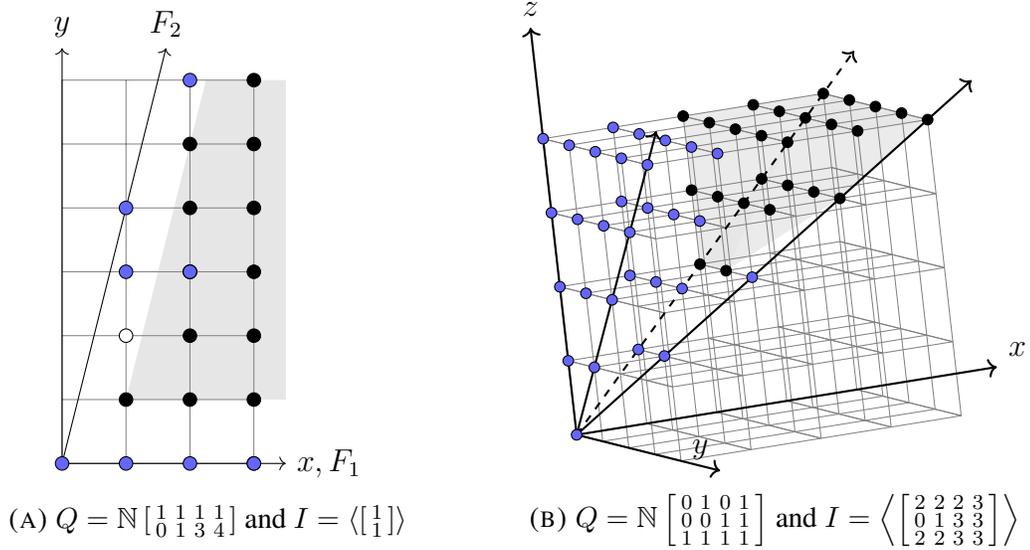
\begin{figure*}[t!]
\centering
\begin{subfigure}[c]{0.45\linewidth}
\centering
\begin{tikzpicture}[scale=0.85]
\fill[idealregioncolor] (1,1) -- (2.25,6) -- (3.5,6) -- (3.5,1) -- cycle ;

  \draw[step=1cm,gray,very thin] (0,0) grid (3,6);
  \draw [<->] (0,6.5) node (yaxis) [above] {$y$}
        |- (3.5,0) node (xaxis) [right] {$x,\facF_{1}$};
\draw[->](0,0) -- (1.625,6.5);

\node[above] at (1.625,6.5) {$\facF_{2}$};
\draw[black,fill=stdmcolor] (0,0) circle (3pt);

\draw[black,fill=stdmcolor] (2,0) circle (3pt);
\draw[black,fill=stdmcolor] (3,0) circle (3pt);
\draw[black,fill=stdmcolor] (3,0) circle (3pt);
\draw[black,fill=stdmcolor] (2,6) circle (3pt);

\foreach \y in {0,3,4}
\draw[black,fill=stdmcolor] (1,\y) circle (3pt);

\foreach \y in {0,3,4,5}
\draw[black,fill=stdmcolor] (2,3) circle (3pt);
\draw[black,fill=white] (1,2) circle (3pt);

\draw[black,fill=idealcolor] (3,6) circle (3pt);
\draw[black,fill=idealcolor] (1,1) circle (3pt);
\foreach \x in {2}
\foreach \y in {1,2,4,5}
\draw[black,fill=idealcolor] (\x,\y) circle (3pt);

\foreach \x in {3}
\foreach \y in {1,...,5}
\draw[black,fill=idealcolor] (\x,\y) circle (3pt);

\end{tikzpicture}
\caption{$Q = \N \left[\begin{smallmatrix}1 & 1 & 1 & 1 \\ 0 & 1 & 3 & 4  \end{smallmatrix}\right]$ and $I=\left\< \left[\begin{smallmatrix} 1 \\ 1 \end{smallmatrix}\right]\right\>$}
\label{fig:frob_coin}
\end{subfigure}
\begin{subfigure}[c]{0.45\linewidth}
        \centering
	 \tdplotsetmaincoords{90}{90} 
        \begin{tikzpicture}[tdplot_main_coords, scale=1]
	\tdplotsetrotatedcoords{40}{-10}{30}
	\fill[idealregioncolor, fill opacity=0.8,tdplot_rotated_coords] (2,0,2) -- (2,1,2) -- (3,3,3) -- (4,4,4) -- (4,0,4) -- cycle;
	\fill[idealregioncolor, fill opacity=0.8,tdplot_rotated_coords] (2,0,2) -- (2,1,2) -- (2,4,4) -- (2,0,4) -- cycle;
	\fill[idealregioncolor, fill opacity=0.8,tdplot_rotated_coords] (2,4,4) -- (2,0,4) -- (4,0,4) -- (4,4,4)-- cycle;

	\foreach \x in {0,...,4}
	{
		\draw[step=1cm,gray,very thin, tdplot_rotated_coords] (\x,0,0) -- (\x,4,0);
		\draw[step=1cm,gray,very thin, tdplot_rotated_coords] (0,\x,0) -- (4,\x,0);
		\draw[step=1cm,gray,very thin, tdplot_rotated_coords] (\x,0,1) -- (\x,4,1);
		\draw[step=1cm,gray,very thin, tdplot_rotated_coords] (0,\x,1) -- (4,\x,1);
		\draw[step=1cm,gray,very thin, tdplot_rotated_coords] (\x,0,2) -- (\x,4,2);
		\draw[step=1cm,gray,very thin, tdplot_rotated_coords] (0,\x,2) -- (4,\x,2);
		\draw[step=1cm,gray,very thin, tdplot_rotated_coords] (\x,0,3) -- (\x,4,3);
		\draw[step=1cm,gray,very thin, tdplot_rotated_coords] (0,\x,3) -- (4,\x,3);
		\draw[step=1cm,gray,very thin, tdplot_rotated_coords] (\x,0,4) -- (\x,4,4);
		\draw[step=1cm,gray,very thin, tdplot_rotated_coords] (0,\x,4) -- (4,\x,4);

		\draw[step=1cm,gray,very thin, tdplot_rotated_coords] (0,\x,0) -- (0,\x,4);
		\draw[step=1cm,gray,very thin, tdplot_rotated_coords] (0,0,\x) -- (0,4,\x);
		\draw[step=1cm,gray,very thin, tdplot_rotated_coords] (1,\x,0) -- (1,\x,4);
		\draw[step=1cm,gray,very thin, tdplot_rotated_coords] (1,0,\x) -- (1,4,\x);
		\draw[step=1cm,gray,very thin, tdplot_rotated_coords] (2,\x,0) -- (2,\x,4);
		\draw[step=1cm,gray,very thin, tdplot_rotated_coords] (2,0,\x) -- (2,4,\x);
		\draw[step=1cm,gray,very thin, tdplot_rotated_coords] (3,\x,0) -- (3,\x,4);
		\draw[step=1cm,gray,very thin, tdplot_rotated_coords] (3,0,\x) -- (3,4,\x);
		\draw[step=1cm,gray,very thin, tdplot_rotated_coords] (4,\x,0) -- (4,\x,4);
		\draw[step=1cm,gray,very thin, tdplot_rotated_coords] (4,0,\x) -- (4,4,\x);
	};
	
	\draw[thick,->,tdplot_rotated_coords] (0,0,0) -- (6,0,0) node[anchor=south west]{$x$}; 
	\draw[thick,->,tdplot_rotated_coords] (0,0,0) -- (0,5.5,0) node[anchor=south east]{$y$}; 
	\draw[thick,->,tdplot_rotated_coords] (0,0,0) -- (0,0,5.5) node[anchor=south]{$z$};
	\draw[thick,->,tdplot_rotated_coords] (0,0,0) -- (0,4.5,4.5);
	\draw[thick,dashed,->,tdplot_rotated_coords] (0,0,0) -- (4.5,0,4.5);
	\draw[thick, ->,tdplot_rotated_coords] (0,0,0) -- (4.5,4.5,4.5);

	\foreach \y in {0,...,4}
	\draw[black,fill=stdmcolor,tdplot_rotated_coords] (0,0,\y) circle (2 pt);
	\draw[black,fill=stdmcolor,tdplot_rotated_coords] (0,1,1) circle (2 pt);
	\draw[black,fill=stdmcolor,tdplot_rotated_coords] (0,1,2) circle (2 pt);
	\draw[black,fill=stdmcolor,tdplot_rotated_coords] (0,2,2) circle (2 pt);
	\draw[black,fill=stdmcolor,tdplot_rotated_coords] (0,1,3) circle (2 pt);
	\draw[black,fill=stdmcolor,tdplot_rotated_coords] (0,2,3) circle (2 pt);
	\draw[black,fill=stdmcolor,tdplot_rotated_coords] (0,3,3) circle (2 pt);
	\draw[black,fill=stdmcolor,tdplot_rotated_coords] (0,1,4) circle (2 pt);
	\draw[black,fill=stdmcolor,tdplot_rotated_coords] (0,2,4) circle (2 pt);
	\draw[black,fill=stdmcolor,tdplot_rotated_coords] (0,3,4) circle (2 pt);
	\draw[black,fill=stdmcolor,tdplot_rotated_coords] (0,4,4) circle (2 pt);

	\foreach \y in {0,...,3}
	\draw[black,fill=stdmcolor,tdplot_rotated_coords] (1,0,1+\y) circle (2 pt);

	\foreach \y in {0,...,3}
	\draw[black,fill=stdmcolor,tdplot_rotated_coords] (1,1+\y,1+\y) circle (2 pt);
	
	\draw[black,fill=stdmcolor,tdplot_rotated_coords] (2,2,2) circle (2 pt);

	\draw[black,fill=stdmcolor,tdplot_rotated_coords] (1,1,2) circle (2 pt);
	\draw[black,fill=stdmcolor,tdplot_rotated_coords] (1,1,3) circle (2 pt);
	\draw[black,fill=stdmcolor,tdplot_rotated_coords] (1,2,3) circle (2 pt);
	\draw[black,fill=stdmcolor,tdplot_rotated_coords] (1,1,4) circle (2 pt);
	\draw[black,fill=stdmcolor,tdplot_rotated_coords] (1,2,4) circle (2 pt);
	\draw[black,fill=stdmcolor,tdplot_rotated_coords] (1,3,4) circle (2 pt);

	\draw[black,fill=idealcolor,tdplot_rotated_coords] (2,0,2) circle (2 pt);
	\draw[black,fill=idealcolor,tdplot_rotated_coords] (2,1,2) circle (2 pt);
	
	\draw[black,fill=idealcolor,tdplot_rotated_coords] (2,0,3) circle (2 pt);
	\draw[black,fill=idealcolor,tdplot_rotated_coords] (2,1,3) circle (2 pt);
	\draw[black,fill=idealcolor,tdplot_rotated_coords] (2,2,3) circle (2 pt);
	\draw[black,fill=idealcolor,tdplot_rotated_coords] (2,3,3) circle (2 pt);
	\draw[black,fill=idealcolor,tdplot_rotated_coords] (3,0,3) circle (2 pt);
	\draw[black,fill=idealcolor,tdplot_rotated_coords] (3,1,3) circle (2 pt);
	\draw[black,fill=idealcolor,tdplot_rotated_coords] (3,2,3) circle (2 pt);
	\draw[black,fill=idealcolor,tdplot_rotated_coords] (3,3,3) circle (2 pt);
	\draw[black,fill=idealcolor,tdplot_rotated_coords] (2,0,4) circle (2 pt);
	\draw[black,fill=idealcolor,tdplot_rotated_coords] (2,1,4) circle (2 pt);
	\draw[black,fill=idealcolor,tdplot_rotated_coords] (2,2,4) circle (2 pt);
	\draw[black,fill=idealcolor,tdplot_rotated_coords] (2,3,4) circle (2 pt);
	\draw[black,fill=idealcolor,tdplot_rotated_coords] (2,4,4) circle (2 pt);
	\draw[black,fill=idealcolor,tdplot_rotated_coords] (3,0,4) circle (2 pt);
	\draw[black,fill=idealcolor,tdplot_rotated_coords] (3,1,4) circle (2 pt);
	\draw[black,fill=idealcolor,tdplot_rotated_coords] (3,2,4) circle (2 pt);
	\draw[black,fill=idealcolor,tdplot_rotated_coords] (3,3,4) circle (2 pt);
	\draw[black,fill=idealcolor,tdplot_rotated_coords] (3,4,4) circle (2 pt);
	\draw[black,fill=idealcolor,tdplot_rotated_coords] (4,0,4) circle (2 pt);
	\draw[black,fill=idealcolor,tdplot_rotated_coords] (4,1,4) circle (2 pt);
	\draw[black,fill=idealcolor,tdplot_rotated_coords] (4,2,4) circle (2 pt);
	\draw[black,fill=idealcolor,tdplot_rotated_coords] (4,3,4) circle (2 pt);
	\draw[black,fill=idealcolor,tdplot_rotated_coords] (4,4,4) circle (2 pt);

	\end{tikzpicture}
	 \caption{$Q = \N \left[\begin{smallmatrix}0&1&0&1\\0&0&1&1\\ 1&1&1&1 \end{smallmatrix}\right]$ and $I=\left\< \left[\begin{smallmatrix} 2& 2 & 2 & 3 \\ 0 & 1 & 3 & 3 \\ 2 & 2 & 3 & 3 \end{smallmatrix}\right]\right\>$}
	 \label{fig:3d}
    \end{subfigure}
\caption{Monomials and ideals in affine semigroups}
\label{fig:affine_semigp}
\end{figure*}

\subsection{Hyperplane arrangements}
\label{subsec:hyperplane}

A polyhedron $\polytope$ can be expressed as an intersection of finitely many
open half-spaces $\{\hplane_{\smalli}^{(+)}\}_{\smalli=1}^{\smalln}$,
where $\arngmt:=\{\hplane_{\smalli}\}_{\smalli=1}^{\smalln}$  is a
collection of hyperplanes and $\smalln \in \N$. This is known as the \emph{$\hplane$-representation} of $\polytope$.
The  collection $\arngmt$ is called the \emph{hyperplane arrangement}
of $\polytope$ in $\R^{d}$. Throughout this article, we assume that
$\arngmt$ is \emph{linear}, meaning that all hyperplanes in $\arngmt$
contain the origin. A \emph{region} is a connected component $\regr$
of $\R^{d}-\bigcup_{\hplane \in \arngmt} \hplane$. Let
$\regionr{\arngmt}$ be a set containing all regions over $\arngmt$.
These are standard notions in combinatorics, see for instance~\cite{Stanley07}.  

Since our arrangement is linear, all regions in $\regionr{\arngmt}$
are unbounded~\cite{Stanley07}, constituting rational polyhedral
cones. Moreover, any region $\regr$ in $\regionr{\arngmt}$ can be
expressed as  
\[
\regr_{\setS}:=\left(\bigcap_{\smalli \in \setS}
  \hplane_{\smalli}^{(+)}\right) \cap \left(\bigcap_{\smalli \in
    [n]\minus\setS} \hplane_{\smalli}^{(-)}\right)
\] 
for a subset $\setS \subseteq [n]$, where $\hplane_{\smalli}^{(-)}$ is
the complement of $\hplane_{\smalli}^{(+)} \cup
\hplane_{\smalli}$. The collection $\regionr{\arngmt}$ is partially
ordered by reverse inclusion; $\regr_{\setS_{1}} \leq
\regr_{\setS_{2}} \text{ if } \setS_{1} \supseteq \setS_{2}.$ The
poset $\regionr{\arngmt}$ is called the \emph{poset of region
  $\regionr{\arngmt}$}. Under the order given by inclusion, this object
has received much attention~\cites{Edelman84,BEZ90}.

We can label faces of the affine semigroup $\affinesemigp$ using the
corresponding indices of its supporting hyperplanes. This gives rise
to a natural embedding
$\facelattice(\affinesemigp) \to
\regionr{\arngmt}$~\cite{Edelman84}*{Lemma 1.3}. To partition
$\R^{n}$, we modify the definition of $\regr_{\setS}$ as follows
\[
\regr_{\setS}:=\overline{\left(\bigcap_{\smalli \in \setS}
    \hplane_{\smalli}^{(+)}\right)} \cap \left(\bigcap_{\smalli \in
    [n]\minus\setS} \hplane_{\smalli}^{(-)}\right),
\]
where $\overline{(-)}$ denotes the closure of the set in the standard $\R^{n}$-topology. 

Similarly, given a region $\regr_{\setS} \in \regionr{\arngmt}$, a
\emph{cumulative region} $\regR_{\setS}$ is the closure of the union
of all regions less than $\regr_{\setS}$ in the partial order. In other words,
\[
\regR_{\setS} := \bigcup_{\substack{\setS' \supseteq \setS \\
    \regr_{\setS'} \in \regionr{\arngmt}}}
\regr_{\setS'}=\overline{\left(\bigcap_{\smalli \in \setS}
    \hplane_{\smalli}^{(+)}\right)}.
\]
The \emph{poset of cumulative regions} $\regionR{\arngmt}$ is the
collection of all cumulative regions ordered by inclusion. As posets,
$\regionR{\arngmt} \cong \regionr{\arngmt}$. It is worth noting that
every element in $\regionR{\arngmt}$ is a rational polyhedral
cone. The following result shows that
$\regionR{\arngmt}$ is the set of all polyhedral cones corresponding
to localizations of $\affinesemigp$. 

\begin{proposition}
\label{pro:localization_cone_structure}
Given a face $\facF \in \facelattice(\affinesemigp)$, let $\setS$ be a set of indices of hyperplanes whose half-space contains $\facF$. Then, $\realize(\affinesemigp -\N\facF) =  \regR_{\setS}$.
\end{proposition}

\begin{proof}
From $\Z\facF \cup \affinesemigp \subset \regR_{\setS}$,
$\realize(\affinesemigp -\N\facF) \subseteq
\regR_{\setS}$. Conversely, for any $\varx \in \regR_{\setS}$, $\<
\vecc_{\smalli}, x\> \geq 0$ when $\smalli \in \setS$. Pick $\vecf \in
\relint(\N\facF)$ and let $\varx'= \varx + \left(\sum_{\smalli \not\in
    \setS} \smalla_{\smalli}\right)\vecf$ where $\smalla_{\smalli}$ is
a non-negative real number such that $\< \vecc_{\smalli},
\varx+\smalla_{\smalli}\vecf\> \geq 0$. Then, $\varx=\varx' +
(\varx-\varx')$ with $\varx' \in \realize{\affinesemigp}$ and
$(\varx-\varx') \in \text{span}(\facF)$. Thus, $\realize(\affinesemigp
-\N\facF) =  \realize(\Z\facF \cup \affinesemigp)\supseteq
\regR_{\setS}$.  
\end{proof}

Let $\catname{Cat}_{\affinesemigp}$ be the poset containing all
localizations of $\affinesemigp$ with an order by inclusion. 
Below, we describe all posets that arise in this section using a
commutative diagram. 
\[
\begin{tikzcd}
\facelattice(\affinesemigp) \ar[r,"\affinesemigp-\N(-)","\cong"'] &\catname{Cat}_{\affinesemigp} \ar[r,hook,"\realize(-)"] & \regionR{\arngmt}\arrow[r,"\cong"] & \regionr{\arngmt}\arrow[r,hook,"\phi"] & (2^{\arngmt})^{\text{op}}  \\
\end{tikzcd}
\]
Note that the embedding $\facelattice(\affinesemigp) \to
\regionr{\arngmt}$ in~\cite{Edelman84}*{Lemma 1.3} is split into the
diagram above. Moreover, all posets are indexed by a subposet of
$(2^{\arngmt})^{op}$, a poset of subsets of $\arngmt$ by reverse
inclusion. The inclusion map $\phi$ returns the set of indices of
positive half-spaces containing the given element. 

\begin{example}[Continuation of \cref{ex:affine_semigp}]
\label{ex:hplane_arngmt}
\
\begin{enumerate}[leftmargin=*]
\item  \label{enum:frob_hplane_arngmt}
Let $\polytope = \realize\affinesemigp$ with $\affinesemigp =
\N\left[\begin{smallmatrix} 1 & 1 & 1& \cdots & 1 \\ 0 & a_{1} & a_{2}
    & \cdots & a_{\smalln-1}\end{smallmatrix}\right]$. Then
$\polytope$ is a 2-dimensional cone with facets (rays)
$\realize\left[\begin{smallmatrix} 1 \\ 0\end{smallmatrix}\right]$ and
$\realize\left[\begin{smallmatrix} 1 \\
    a_{\smalln-1}\end{smallmatrix}\right]$. Hence $\arngmt =
\{\hplane_{1}:=\R\left[\begin{smallmatrix} 1 \\
    0\end{smallmatrix}\right],\hplane_{2}:=\R\left[\begin{smallmatrix}
    1 \\ a_{\smalln-1}\end{smallmatrix}\right]\}$. Since $\polytope$
is a homogenization of the 1-simplex, $\regionr{\arngmt},
\regionR{\arngmt}, \facelattice(\affinesemigp)$ and
$\catname{Cat}_{\affinesemigp}$ are all isomorphic as posets.  
\[
\begin{tabular}{|c|c|c|c|c|}
\hline
 $\facelattice(\affinesemigp)$ & $\catname{Cat}_{Q}$ & $\regionR{\arngmt}$ & $\regionr{\arngmt}$ & $2^{\arngmt}$  \\
\hline
 $0$ & $\affinesemigp$ & $\polytope$ & $\polytope=\hplane_{1}^{(+)} \cap \hplane_{2}^{(+)}$ & $\{\hplane_{1},\hplane_{2}\}$  \\
\hline
 $\facF_{1}$ & $\affinesemigp - \N \facF_{1}$ & $\hplane_{1}^{(+)} = \{ y \geq 0\}$ & $\hplane_{1}^{(+)} \cap \hplane_{2}^{(-)} = \{ y \geq 0, y> a_{\smalln-1}x \}$ & $\{\hplane_{1}\}$  \\
\hline
 $\facF_{2}$ & $\affinesemigp - \N \facF_{2}$ & $\hplane_{2}^{(+)} = \{ y\leq a_{\smalln-1}x\}$ & $\hplane_{1}^{(-)} \cap \hplane_{2}^{(+)} = \{ y < 0, y\leq a_{\smalln-1}x \}$ & $\{\hplane_{2}\}$  \\
\hline
 $Q$ & $\Z^{2}$ & $\R^{2}$ & $\hplane_{1}^{(-)} \cap \hplane_{2}^{(-)} = \{ y < 0, y> a_{\smalln-1}x \}$ & $\varnothing$  \\
\hline
\end{tabular}
\]
\item \label{enum:3d_hplane_arngmt}
Given $\affinesemigp =\N \left[\begin{smallmatrix}0 & 1 & 1 & 0 \\ 0 & 0 & 1 & 1 \\ 1 & 1 & 1 & 1 \end{smallmatrix}\right],$ denote $\veca_{i}$ for the $i$-th column of $\left[\begin{smallmatrix}0 & 1 & 1 & 0 \\ 0 & 0 & 1 & 1 \\ 1 & 1 & 1 & 1 \end{smallmatrix}\right].$  Then, as seen below, we may represent facets $\facF_{\smalli}$ and the hyperplane arrangement $\arngmt=\{ \hplane_{1},\hplane_{2},\hplane_{3},\hplane_{4}\}$.
\begin{align*}
\facF_{1}&:= \langle \veca_{1},\veca_{2}\rangle, & \facF_{2}&:= \langle \veca_{2},\veca_{3}\rangle, & \facF_{3}&:= \langle \veca_{3},\veca_{4}\rangle, &\facF_{4}&:= \langle \veca_{4},\veca_{1}\rangle \\
\hplane_{1}^{(+)}&:= \{ y > 0\}, & \hplane_{2}^{(+)}&:= \{ z > x\}, & \hplane_{3}^{(+)}&:= \{ z > y\},& \hplane_{4}^{(+)}&:= \{ x > 0\}.
\end{align*}
Observe that $\realize{(\affinesemigp-\facF_{\smalli})}= \regR_{\{\smalli\}}$ for any $\smalli \in [n]$ and
\begin{align*}
\realize{(\affinesemigp-\langle \veca_{1}\rangle)} &=\regR_{1,4},& 
\realize{(\affinesemigp-\langle \veca_{2}\rangle)} &= \regR_{1,2}, \\
\realize{(\affinesemigp-\langle \veca_{3}\rangle)} & =\regR_{2,3},&
\realize{(\affinesemigp-\langle \veca_{4}\rangle)} &=\regR_{3,4}.
\end{align*}
Thus, $\regionR{\arngmt}\supsetneq
\realize{(\catname{Cat}_{\affinesemigp})}$, for example, because
localization cannot generate affine semigroups in
$\regR_{1,2,3}$. This illustrates the nontrivial injection
$\regionr{\arngmt} \cong \regionR{\arngmt} \supsetneq
\facelattice(\affinesemigp) \cong\catname{Cat}_{\affinesemigp}$,
described in~\cref{fig:3d_hasse_diagram}. 
\end{enumerate}
\end{example}

\begin{figure*}[t!]
\centering
\begin{tikzpicture}[scale=0.8]
  \node (max) at (0,4) {$\Z\affinesemigp$};
  \node (abd) at (-3,2) {$\affinesemigp-\facF_{1}$};
  \node (abc) at (-1,2) {$\affinesemigp-\facF_{2}$};
  \node (bcd) at (1,2) {$\affinesemigp-\facF_{3}$};
  \node (acd) at (3,2) {$\affinesemigp-\facF_{4}$};
  \node (ad) at (-3,0) {$\affinesemigp-\langle \veca_{1}\rangle$};
  \node (ab) at (-1,0) {$\affinesemigp-\langle \veca_{2}\rangle$};
  \node (bc) at (1,0) {$\affinesemigp-\langle \veca_{3}\rangle$};
  \node (cd) at (3,0) {$\affinesemigp-\langle \veca_{4}\rangle$};
  \node (min) at (0,-4) {$\affinesemigp$};
  \draw (max) -- (abc) (max) -- (abd) (max) -- (acd) (max) -- (bcd);
  \draw (abc)--(ab) (abc) -- (bc) (bcd) -- (bc) (bcd) -- (cd) (abd) -- (ab) (abd) -- (ad);
\draw[preaction={draw=white, -,line width=6pt}] (acd) -- (ad) (acd) --(cd);
  \draw (ad) -- (min) (ab) -- (min)  (bc)--(min) (cd) -- (min);

\draw[->] (4,0) -- (6,0);
\node[above] at (5,0) {$\realize{(-)}$};

  \node (zmax) at (10,4) {$\Z\affinesemigp$};
  \node (zabd) at (7,2) {$\regR_{1}$};
  \node (zabc) at (9,2) {$\regR_{2}$};
  \node (zbcd) at (11,2) {$\regR_{3}$};
  \node (zacd) at (13,2) {$\regR_{4}$};
  \node (zad) at (7,0) {$\regR_{1,4}$};
  \node (zab) at (9,0) {$\regR_{1,2}$};
  \node (zbc) at (11,0) {$\regR_{2,3}$};
  \node (zcd) at (13,0) {$\regR_{3,4}$};
  \node (zd) at (7,-2) {$\regR_{1,3,4}$};
  \node (za) at (9,-2) {$\regR_{1,2,4}$};
  \node (zb) at (11,-2) {$\regR_{1,2,3}$};
  \node (zc) at (13,-2) {$\regR_{2,3,4}$};
  \node (zmin) at (10,-4) {$\affinesemigp$};
  \draw (zmax) -- (zabc) (zmax) -- (zabd) (zmax) -- (zacd) (zmax) -- (zbcd);
  \draw (zabc)--(zab) (zabc) -- (zbc) (zbcd) -- (zbc) (zbcd) -- (zcd) (zabd) -- (zab) (zabd) -- (zad);
\draw[preaction={draw=white, -,line width=6pt}] (zacd) -- (zad) (zacd) --(zcd);
  \draw (zad) -- (za) (zad) -- (zd) (zab) -- (za) (zab) -- (zb) (zbc)--(zb) (zbc) -- (zc);
    \draw[preaction={draw=white, -,line width=6pt}] (zcd)--(zd) (zcd) -- (zc);
      \draw (zmin) -- (za) (zmin) -- (zb) (zmin) -- (zc) (zmin) -- (zd);
\end{tikzpicture}
\caption{Hasse diagrams of $\catname{Cat}_{\affinesemigp}$ and $\regionR{\arngmt}$ from~\cref{enum:3d_hplane_arngmt} of~\cref{ex:hplane_arngmt}}
\label{fig:3d_hasse_diagram}
\end{figure*}

\subsection{Ishida complex}
\label{subsec:Ishida}

Using the polyhedral cone structure of an affine semigroup, the Ishida
complex was developed to compute the local cohomology of modules over
pointed affine semigroup rings supported on the graded maximal
ideal~\cite{MR977758}. Given a pointed affine semigroup
$\affinesemigp$, $\realize{\affinesemigp}$ has a \emph{transverse
  section} $\crosssection$~\cite{Ziegler95}*{Exercise 2.19}, which is
a closed polytope generated by intersecting $\realize\affinesemigp$
with a hyperplane $\hplane$ that intersects all unbounded faces of
$\realize{\affinesemigp}$. This results in the following canonical
isomorphism $\virtual{-}:\facelattice(\crosssection) \to
\facelattice(\affinesemigp)$ where $\virtual{\facF}$ is the minimal
face of $\affinesemigp$ such that $\realize\virtual{\facF} \supseteq
\facF$. The vertex of $\affinesemigp$ corresponds to the
(-1)-dimensional face $\varnothing$ of
$\crosssection$. $\crosssection$ is a CW complex with an incidence
function 
\[
\incidence:
\bigoplus_{\smalli=-1}^{\dimd-1}\facelattice(\crosssection)
^{\smalli}\times  \facelattice(\crosssection)^{\smalli-1}   \to
\{0,\pm1 \}.
\] 
in which $\incidence$ is determined by an orientation of
$\crosssection$ and has a nonzero value when two faces are
incident~\cite{Massey80}*{IV.~\S~5}.   

\begin{definition}[Ishida complex~\cite{MR977758}]
\label{def:ishida_cpx}
Let $\maxid$ be the maximal monomial ideal of
$\field[\affinesemigp]$. The set of all $k$-dimensional faces in
$\facelattice(\crosssection)$ is denoted by
$\facelattice(\crosssection)^{k}$. Let $\ishida^{\bullet}$ be the
chain complex  
\[
\begin{tikzcd}
\ishida^{\bullet}: 0\arrow[r] & \ishida^{0} \arrow[r,"\partial"] &\ishida^{1} \arrow[r,"\partial"] & \cdots \arrow[r,"\partial"] & \ishida^{d} \arrow[r,"\partial"] & 0, \qquad L^{\smallk} := \bigoplus\limits_{\facF \in \facelattice(\crosssection)^{\smallk-1}}\field[\affinesemigp-\N\virtual{\facF}]
\end{tikzcd}
\]
where the differential $\partial :L^{\smallk} \to L^{\smallk+1}$ is induced by componentwise map $\partial_{\facF,\facG}$ with $\facF \in  \facelattice(\crosssection)^{\smallk-1}$, $\facG \in  \facelattice(\crosssection)^{\smallk}$ such that
\[
\partial_{\facF,\facG}: \field[\affinesemigp-\N\virtual{\facF}] \to
\field[\affinesemigp-\N\virtual{\facG}] \text{ to be } \begin{cases} 0
  & \text{ if } \facF \not\subset \facG \\
  \incidence(\facF,\facG)\cdot\nat & \text{ if } \facF \subset
  \facG  \end{cases}
\]
with $\nat$, the canonical injection
$\field[\affinesemigp-\N\virtual{\facF}] \to
\field[\affinesemigp-\N\virtual{\facG}]$ when $\facF \subseteq
\facG$. We say that $\ishida^{\bullet}
\otimes_{\field[\affinesemigp]}\module$ is the \emph{Ishida complex}
of a $\field[\affinesemigp]$-module $\module$ supported at the maximal
monomial ideal. 
\end{definition}

The Ishida complex indeed computes local cohomology as follows.

\begin{theorem}[\cite{MR977758}*{Theorem 6.2.5}]
\label{thm:Ishida}
For any graded $\field[\affinesemigp]$-module $\module$, and all $\smallk \geq 0$,
\[
\localcoho{\maxid}{\smallk}{\module} \cong
H^{\smallk}(\ishida^{\bullet} \tensor{\field[\affinesemigp]}\module).
\]
\end{theorem}

\begin{example}[Continuation of \cref{ex:affine_semigp}]
\label{ex:ishida}
\
\begin{enumerate}[leftmargin=*]
\item \label{enum:frob_ishida}
Given $\affinesemigp/\idI = \N\left[\begin{smallmatrix} 1 & 1 & 1 & 1
    \\ 0 & 1 & 3 &
    4\end{smallmatrix}\right]/\left\<\left[\begin{smallmatrix} 1 \\
      1\end{smallmatrix}\right]\right\>$, let
$S=\field[\affinesemigp]/\midI$, where $\midI$ is a monomial ideal
corresponding to $\idI$. The transverse section of
$\realize{\affinesemigp}$ is a line segment with vertices
$\facF_{1}=\left[\begin{smallmatrix} 1 \\ 0 \end{smallmatrix}\right]$
and $\facF_{2}=\left[\begin{smallmatrix} 1 \\
    4 \end{smallmatrix}\right]$ respectively. Thus, the corresponding
Ishida complex of $S$ with the maximal ideal support is 
\[
\ishida^{\bullet}:0 \to S \to S_{x} \oplus S_{xy^4} \to 0 \to 0
\]
\item \label{enum:3d_ishida}
Given $\affinesemigp/\idI =\N \left[\begin{smallmatrix}0 & 1 & 1 & 0
    \\ 0 & 0 & 1 & 1 \\ 1 & 1 & 1 &
    1 \end{smallmatrix}\right]/\left\<\left[\begin{smallmatrix} 2& 2 &
      2 & 3 \\ 0 & 1 & 3 & 3 \\ 2 & 2 & 3 &
      3 \end{smallmatrix}\right]\right\>$, let
$S=\field[\affinesemigp]/\midI$, where $\midI$ is a monomial ideal
corresponding to $\idI$. $\realize{\affinesemigp}$'s transverse
section is rectangular. Due to the fact that all other localizations
of $S$ are zero except for the localizations by monomial prime ideals
corresponding to $\virtual{\veca_{1}}:=\left[\begin{smallmatrix} 0 \\
    0
    \\1 \end{smallmatrix}\right],\virtual{\veca_{4}}:=\left[\begin{smallmatrix}0
    \\ 1 \\ 1 \end{smallmatrix}\right],$ and
$\facF_{4}:=\left[\begin{smallmatrix} 0 & 0 \\ 0 & 1 \\ 1 &
    1 \end{smallmatrix}\right]$, the Ishida complex of $S$ with the
maximal ideal support is as follows;  
\[
\ishida^{\bullet}:0 \to S \to S_{z} \oplus S_{yz} \to S_{z,yz} \to 0
\]
\end{enumerate}
\end{example}

\section{Multidegrees and localization for graded \texorpdfstring{$\field[\N\genset]$}--module}
\label{sec:stdpairs}

Let $\affinesemigp = \N\genset$ be an affine semigroup. We know that the
faces of $\affinesemigp$ govern the localizations of any $\Z^{\dimd}$-graded
$\field[\affinesemigp]$-module $\module$ by monomials. In this section, we examine the
effect of localization on the supporting multidegrees of $\module$,
defined as follows.

\begin{definition}
  \label{def:degreePairs}
Let $\module$ be a $\Z^{\dimd}$-graded
$\field[\affinesemigp]$-module.
\begin{enumerate}
\item The \emph{degree set} of $\module$ is
defined to be
\[
\deg(\module):=\{ \veca \in \Z^{\dimd}\mid \module_{\veca} \neq
0\}.\]
\item
A \emph{proper pair} of $\module$ is a pair $(\veca,\facF)$
where $\veca \in \affinesemigp$ and $\facF \in
\facelattice(\affinesemigp)$ such that $\veca +\N\facF \subseteq
\deg(\module)$.
\item
If $(\veca,\facF)$ and $(\vecb,\facG)$ are proper
pairs, we say $(\veca,\facF) < (\vecb,\facG)$ if $\veca + \N\facF
\subseteq \vecb +\N\facG$. A proper pair $(\veca,\facF)$ of $M$ is
called a \emph{degree pair} of $\module$ if it is maximal among
proper pairs in this partial order.
\end{enumerate}
\end{definition}

\begin{lemma}
\label{lem:finiteness_of_degree_pairs}
Any finitely generated $\Z^{\dimd}$-graded
$\field[\affinesemigp]$-module $\module$ has finitely many degree
pairs. 
\end{lemma}

\begin{proof}
Let $0=\module_{0} \subset \module_{1} \subset \cdots \module_{l}
=\module$ be a chain of submodules of $\module$ such that
$\module_{i-1}/\module_{i}\cong \field[\N\genset]/P_{i}$ where $P_{i}$
is a graded prime ideal of $\field[\N\genset]$.

Thus, $\deg\left(\module_{i}/\module_{i-1}\right) =
\veca_{i}+\N\facF_{i}$ for some $\veca_{i} \in \Z^{\dimd}$ and a face
$\facF_{i}$ corresponding to $P_{i}$.
To see that
\[
\deg\left(\module\right) = \deg\left(\module_{0} \oplus
  \module_{1}/\module_{0} \oplus \cdots \oplus
  \module_{l}/\module_{l-1}\right) =
\bigcup_{i=1}^{l}\veca_{i}+\N\facF_{i},
\]
note that for any homogeneous element $m \in \module$, there exists
$i$ such that $\overline{m} \in \module_{i}/\module_{i-1}$ is
nonzero. Hence, $\deg(m) \in \veca_{i}+\N\facF_{i}$. Conversely, we
may lift any graded element in the direct sum to $\module$.
This says we have a finite pair cover, from here to finitely 
many standard pairs as demonstrated by~\cite{STDPAIR}*{Theorem 3.16}.
\end{proof}

Two degree pairs $(\veca,\facF)$ and $(\vecb,\facF)$ with the same
face $\facF$ \emph{overlap} if the intersection $(\veca +
\N\facF)\cap(\vecb+\N\facF)$ is nonempty. Overlapping is an
equivalence relation. The \emph{overlap class} $[\veca,\facF]$ is the
equivalence class containing the degree pair $(\veca,\facF)$. We
define $\degp(\module)$ (resp. $\degpo(\module)$) as the set of all
(resp. overlap classes of) degree pairs of $\module$. These
definitions are based on~\cite{STDPAIR}*{Definition~3.1 and 3.2},
which apply to general (not necessarily pointed) affine semigroup
rings, but only in the monomial ideal case. 

\begin{example}[Standard pairs and void pairs]
\label{ex:std_void}
\
\begin{enumerate}[leftmargin=*]
\item \label{enum:stdpairs}
Let $\midI$ be a monomial ideal of $\field[\N\genset]$. The degree
pairs of $\module =\field[\N\genset]/\midI$ are the \emph{standard
  pairs} of $\midI$ introduced in~\cite{STV95} and generalized
in~\cite{STDPAIR}*{Definition~3.1 and 3.2}. 
\item \label{enum:voidpairs} 
Given an affine semigroup $\affinesemigp:=\N\genset$, the
\emph{saturation of $\affinesemigp$} is $\affinesemigp_{\text{sat}} = \Z^{\dimd} \cap
\realize{\affinesemigp}$. It is known that the affine semigroup ring
corresponding to the saturation of of $\affinesemigp$ is the
normalization of $\field[\affinesemigp]$.
The set of \emph{holes} of $\affinesemigp$ is defined to be the
difference $\affinesemigp_{\text{sat}} \smallsetminus \affinesemigp$.

\medskip
\noindent The set of holes of 
$\affinesemigp$ is also the degree set
$\deg(\field[\affinesemigp_{\text{sat}}]/\field[\affinesemigp])$. As the $\field[\affinesemigp]$-module
$\module=\field[\affinesemigp_{\text{sat}}]/\field[\affinesemigp]$ is finitely generated by
Noether's normalization lemma, applying ~\cref{lem:finiteness_of_degree_pairs} provides an
alternative algebraic proof of the well-known
combinatorial result~\cite{HTY09}. Namely, the set of holes of a
semigroup $\affinesemigp$ is a finite union of translates of faces of $\affinesemigp$.
Later
on, we refer to degree pairs of $\module$ as \emph{void pairs}.
\end{enumerate}
\end{example}

\begin{example}[Continuation of \cref{ex:ishida}]
\label{ex:stdpairs}
\
\begin{enumerate}[leftmargin=*]
\item \label{enum:frob_stdpairs}
Degree pairs of $\module:=\field[x,xy,xy^3,xy^4]/\<xy\>$ are
\[
\text{(green) }\textcolor{green!80!black}{\left( \left[\begin{smallmatrix} 2\\ 3\end{smallmatrix}\right], \varnothing \right)},\text{(blue) }\textcolor{blue}{\left( \left[\begin{smallmatrix} 0\\ 0\end{smallmatrix}\right], F_{1} \right)}, \text{(red) }\textcolor{red}{\left( \left[\begin{smallmatrix} 0\\ 0\end{smallmatrix}\right], F_{2} \right),\left( \left[\begin{smallmatrix} 1\\ 3\end{smallmatrix}\right], F_{2} \right), \left( \left[\begin{smallmatrix} 2\\ 6\end{smallmatrix}\right], F_{2} \right).}
\]
In \cref{fig:frob_stdpairs}, these are indicated by a green (dotted)
circle, a blue (dashed) line, and red (straight) lines. Each of the
standard pairs forms an overlap class. 
\item \label{enum:3d_stdpairs}
Let $\genset:=\left[\begin{smallmatrix}0 & 1 & 1 & 0 \\ 0 & 0 & 1 & 1 \\ 1 & 1 & 1 & 1 \end{smallmatrix}\right]$. Degree pairs of $\module:=\field\left[\N\genset\right]/\midI$ where $\midI = \left\<x^{2}z^{2}, x^{2}yz^{2},x^{2}y^{3}z^{3},x^{3}y^{3}z^{3}\right\>$ is a monomial ideal of $\field[z,xz,xyz,yz] \cong \field\left[\N\genset\right]$ are
\[
\text{(green) }\textcolor{green!80!black}{\left( \left[\begin{smallmatrix} 2\\ 2 \\ 2\end{smallmatrix}\right], \varnothing \right)}, \text{(blue) } \textcolor{blue}{\left( \left[\begin{smallmatrix} 0\\ 0 \\ 0\end{smallmatrix}\right], F_{4} \right)}, \text{(yellow) }\textcolor{yellow!80!black}{\left( \left[\begin{smallmatrix} 1\\ 0 \\ 1\end{smallmatrix}\right], F_{4} \right)}, \text{(red) }\textcolor{red}{\left( \left[\begin{smallmatrix} 1\\ 1 \\ 1\end{smallmatrix}\right], F_{4} \right)}.
\]
In~\cref{fig:3d_stdpairs}, these are indicated by a green circle, a
blue triangle (in $zy$-plane), a yellow triangle, and a red triangle
(in $x=1$ plane), respectively. As illustrated
in~\cref{fig:3d_stdpairs}, the yellow and red triangles represented by
$\left( \left[\begin{smallmatrix} 1\\ 0 \\ 1\end{smallmatrix}\right],
  F_{4} \right)$ and $\left( \left[\begin{smallmatrix} 1\\ 1 \\
      1\end{smallmatrix}\right], F_{4} \right)$ overlap. Hence,
the overlap classes of $\midI$ are 
\[
\left\{\left( \left[\begin{smallmatrix} 2\\ 2 \\
        2\end{smallmatrix}\right], \varnothing \right)\right\},
\left\{\left( \left[\begin{smallmatrix} 0\\ 0 \\
        0\end{smallmatrix}\right], F_{4} \right)\right\},
\left\{\left( \left[\begin{smallmatrix} 1\\ 0 \\
        1\end{smallmatrix}\right], F_{4} \right),\left(
    \left[\begin{smallmatrix} 1\\ 1 \\ 1\end{smallmatrix}\right],
    F_{4} \right)\right\}.
\]
Notably, the union $\left( \left[\begin{smallmatrix} 1\\ 0 \\ 1\end{smallmatrix}\right] +\N F_{4} \right)\cup\left( \left[\begin{smallmatrix} 1\\ 1 \\ 1\end{smallmatrix}\right]+\N F_{4} \right)$ is a subset of translates of faces represented by standard pairs of $\midI_{\facF_{4}}$, the localization of $\midI$ by a face $\facF_{4}$.
\end{enumerate}
\end{example}

\begin{figure*}[t!]
\centering
\begin{subfigure}[c]{0.45\linewidth}
\centering
\begin{tikzpicture}[scale=0.85]

\fill[idealregioncolor] (1,1) -- (2.25,6) -- (3.5,6) -- (3.5,1) -- cycle ;
  \draw[step=1cm,gray,very thin] (0,0) grid (3,6);
  \draw [<->] (0,6.5) node (yaxis) [above] {$y$}
        |- (3.5,0) node (xaxis) [right] {$x,\facF_{1}$};

\draw[->](0,0) -- (1.625,6.5);

\node[above] at (1.625,6.5) {$\facF_{2}$};
\draw[black,fill=stdmcolor] (0,0) circle (3pt);

\draw[black,fill=stdmcolor] (2,0) circle (3pt);
\draw[black,fill=stdmcolor] (3,0) circle (3pt);
\draw[black,fill=stdmcolor] (3,0) circle (3pt);
\draw[black,fill=stdmcolor] (2,6) circle (3pt);

\foreach \y in {0,3,4}
\draw[black,fill=stdmcolor] (1,\y) circle (3pt);

\foreach \y in {0,3,4,5}
\draw[black,fill=stdmcolor] (2,3) circle (3pt);
\draw[black,fill=white] (1,2) circle (3pt);

\draw[black,fill=idealcolor] (3,6) circle (3pt);
\draw[black,fill=idealcolor] (1,1) circle (3pt);
\foreach \x in {2}
\foreach \y in {1,2,4,5}
\draw[black,fill=idealcolor] (\x,\y) circle (3pt);

\foreach \x in {3}
\foreach \y in {1,...,5}
\draw[black,fill=idealcolor] (\x,\y) circle (3pt);

\draw[red,ultra thick,->] (0,0) -- (1.625,6.5);
\draw[red,ultra thick,->] (1,3) -- (1.875,6.5);
\draw[red,ultra thick,->] (2,6) -- (2.125,6.5);
\draw[blue,loosely dashed,ultra thick,->] (0,0) -- (3.5,0);
\draw[green!80!black,densely dotted,ultra thick] (2,3) circle (10 pt);

\end{tikzpicture}
\caption{$Q = \N \left[\begin{smallmatrix}1 & 1 & 1 & 1 \\ 0 & 1 & 3 & 4  \end{smallmatrix}\right]$ and $\deg(\midI)=\left\< \left[\begin{smallmatrix} 1 \\ 1 \end{smallmatrix}\right]\right\>$}
\label{fig:frob_stdpairs}
\end{subfigure}
\begin{subfigure}[c]{0.45\linewidth}
        \centering
	 \tdplotsetmaincoords{90}{90} 
        \begin{tikzpicture}[tdplot_main_coords, scale=1]
	\tdplotsetrotatedcoords{40}{-10}{30}
	\fill[idealregioncolor, fill opacity=0.8,tdplot_rotated_coords] (2,0,2) -- (2,1,2) -- (3,3,3) -- (4,4,4) -- (4,0,4) -- cycle;
	\fill[idealregioncolor, fill opacity=0.8,tdplot_rotated_coords] (2,0,2) -- (2,1,2) -- (2,4,4) -- (2,0,4) -- cycle;
	\fill[idealregioncolor, fill opacity=0.8,tdplot_rotated_coords] (2,4,4) -- (2,0,4) -- (4,0,4) -- (4,4,4)-- cycle;

	\fill[yellow!80, fill opacity=0.8,tdplot_rotated_coords] (1,0,1) -- (1,3,4) -- (1,0,4) -- cycle;
	\fill[red!80, fill opacity=0.8,tdplot_rotated_coords] (1,1,1) -- (1,4,4) -- (1,1,4) -- cycle;
	\fill[orange!80, fill opacity=0.8,tdplot_rotated_coords] (1,1,2) -- (1,1,4) -- (1,3,4) -- cycle;
	\fill[blue!20,fill opacity=0.8,tdplot_rotated_coords] (0,0,0) -- (0,4,4) -- (0,0,4) -- cycle;
	\draw[green!80!black,ultra thick,tdplot_rotated_coords] (2,2,2) circle (6 pt);
	\foreach \x in {0,...,4}
	{
		\draw[step=1cm,gray,very thin, tdplot_rotated_coords] (\x,0,0) -- (\x,4,0);
		\draw[step=1cm,gray,very thin, tdplot_rotated_coords] (0,\x,0) -- (4,\x,0);
		\draw[step=1cm,gray,very thin, tdplot_rotated_coords] (\x,0,1) -- (\x,4,1);
		\draw[step=1cm,gray,very thin, tdplot_rotated_coords] (0,\x,1) -- (4,\x,1);
		\draw[step=1cm,gray,very thin, tdplot_rotated_coords] (\x,0,2) -- (\x,4,2);
		\draw[step=1cm,gray,very thin, tdplot_rotated_coords] (0,\x,2) -- (4,\x,2);
		\draw[step=1cm,gray,very thin, tdplot_rotated_coords] (\x,0,3) -- (\x,4,3);
		\draw[step=1cm,gray,very thin, tdplot_rotated_coords] (0,\x,3) -- (4,\x,3);
		\draw[step=1cm,gray,very thin, tdplot_rotated_coords] (\x,0,4) -- (\x,4,4);
		\draw[step=1cm,gray,very thin, tdplot_rotated_coords] (0,\x,4) -- (4,\x,4);

		\draw[step=1cm,gray,very thin, tdplot_rotated_coords] (0,\x,0) -- (0,\x,4);
		\draw[step=1cm,gray,very thin, tdplot_rotated_coords] (0,0,\x) -- (0,4,\x);
		\draw[step=1cm,gray,very thin, tdplot_rotated_coords] (1,\x,0) -- (1,\x,4);
		\draw[step=1cm,gray,very thin, tdplot_rotated_coords] (1,0,\x) -- (1,4,\x);
		\draw[step=1cm,gray,very thin, tdplot_rotated_coords] (2,\x,0) -- (2,\x,4);
		\draw[step=1cm,gray,very thin, tdplot_rotated_coords] (2,0,\x) -- (2,4,\x);
		\draw[step=1cm,gray,very thin, tdplot_rotated_coords] (3,\x,0) -- (3,\x,4);
		\draw[step=1cm,gray,very thin, tdplot_rotated_coords] (3,0,\x) -- (3,4,\x);
		\draw[step=1cm,gray,very thin, tdplot_rotated_coords] (4,\x,0) -- (4,\x,4);
		\draw[step=1cm,gray,very thin, tdplot_rotated_coords] (4,0,\x) -- (4,4,\x);
	};
	
	\draw[thick,->,tdplot_rotated_coords] (0,0,0) -- (6,0,0) node[anchor=south west]{$x$}; 
	\draw[thick,->,tdplot_rotated_coords] (0,0,0) -- (0,5.5,0) node[anchor=south east]{$y$}; 
	\draw[thick,->,tdplot_rotated_coords] (0,0,0) -- (0,0,5.5) node[anchor=south]{$z$};
	\draw[thick,->,tdplot_rotated_coords] (0,0,0) -- (0,4.5,4.5);
	\draw[thick,dashed,->,tdplot_rotated_coords] (0,0,0) -- (4.5,0,4.5);
	\draw[thick, ->,tdplot_rotated_coords] (0,0,0) -- (4.5,4.5,4.5);

	\foreach \y in {0,...,4}
	\draw[black,fill=stdmcolor,tdplot_rotated_coords] (0,0,\y) circle (2 pt);
	\draw[black,fill=stdmcolor,tdplot_rotated_coords] (0,1,1) circle (2 pt);
	\draw[black,fill=stdmcolor,tdplot_rotated_coords] (0,1,2) circle (2 pt);
	\draw[black,fill=stdmcolor,tdplot_rotated_coords] (0,2,2) circle (2 pt);
	\draw[black,fill=stdmcolor,tdplot_rotated_coords] (0,1,3) circle (2 pt);
	\draw[black,fill=stdmcolor,tdplot_rotated_coords] (0,2,3) circle (2 pt);
	\draw[black,fill=stdmcolor,tdplot_rotated_coords] (0,3,3) circle (2 pt);
	\draw[black,fill=stdmcolor,tdplot_rotated_coords] (0,1,4) circle (2 pt);
	\draw[black,fill=stdmcolor,tdplot_rotated_coords] (0,2,4) circle (2 pt);
	\draw[black,fill=stdmcolor,tdplot_rotated_coords] (0,3,4) circle (2 pt);
	\draw[black,fill=stdmcolor,tdplot_rotated_coords] (0,4,4) circle (2 pt);

	\foreach \y in {0,...,3}
	\draw[black,fill=stdmcolor,tdplot_rotated_coords] (1,0,1+\y) circle (2 pt);

	\foreach \y in {0,...,3}
	\draw[black,fill=stdmcolor,tdplot_rotated_coords] (1,1+\y,1+\y) circle (2 pt);
	
	\draw[black,fill=stdmcolor,tdplot_rotated_coords] (2,2,2) circle (2 pt);

	\draw[black,fill=stdmcolor,tdplot_rotated_coords] (1,1,2) circle (2 pt);
	\draw[black,fill=stdmcolor,tdplot_rotated_coords] (1,1,3) circle (2 pt);
	\draw[black,fill=stdmcolor,tdplot_rotated_coords] (1,2,3) circle (2 pt);
	\draw[black,fill=stdmcolor,tdplot_rotated_coords] (1,1,4) circle (2 pt);
	\draw[black,fill=stdmcolor,tdplot_rotated_coords] (1,2,4) circle (2 pt);
	\draw[black,fill=stdmcolor,tdplot_rotated_coords] (1,3,4) circle (2 pt);

	\draw[black,fill=idealcolor,tdplot_rotated_coords] (2,0,2) circle (2 pt);
	\draw[black,fill=idealcolor,tdplot_rotated_coords] (2,1,2) circle (2 pt);
	
	\draw[black,fill=idealcolor,tdplot_rotated_coords] (2,0,3) circle (2 pt);
	\draw[black,fill=idealcolor,tdplot_rotated_coords] (2,1,3) circle (2 pt);
	\draw[black,fill=idealcolor,tdplot_rotated_coords] (2,2,3) circle (2 pt);
	\draw[black,fill=idealcolor,tdplot_rotated_coords] (2,3,3) circle (2 pt);
	\draw[black,fill=idealcolor,tdplot_rotated_coords] (3,0,3) circle (2 pt);
	\draw[black,fill=idealcolor,tdplot_rotated_coords] (3,1,3) circle (2 pt);
	\draw[black,fill=idealcolor,tdplot_rotated_coords] (3,2,3) circle (2 pt);
	\draw[black,fill=idealcolor,tdplot_rotated_coords] (3,3,3) circle (2 pt);
	\draw[black,fill=idealcolor,tdplot_rotated_coords] (2,0,4) circle (2 pt);
	\draw[black,fill=idealcolor,tdplot_rotated_coords] (2,1,4) circle (2 pt);
	\draw[black,fill=idealcolor,tdplot_rotated_coords] (2,2,4) circle (2 pt);
	\draw[black,fill=idealcolor,tdplot_rotated_coords] (2,3,4) circle (2 pt);
	\draw[black,fill=idealcolor,tdplot_rotated_coords] (2,4,4) circle (2 pt);
	\draw[black,fill=idealcolor,tdplot_rotated_coords] (3,0,4) circle (2 pt);
	\draw[black,fill=idealcolor,tdplot_rotated_coords] (3,1,4) circle (2 pt);
	\draw[black,fill=idealcolor,tdplot_rotated_coords] (3,2,4) circle (2 pt);
	\draw[black,fill=idealcolor,tdplot_rotated_coords] (3,3,4) circle (2 pt);
	\draw[black,fill=idealcolor,tdplot_rotated_coords] (3,4,4) circle (2 pt);
	\draw[black,fill=idealcolor,tdplot_rotated_coords] (4,0,4) circle (2 pt);
	\draw[black,fill=idealcolor,tdplot_rotated_coords] (4,1,4) circle (2 pt);
	\draw[black,fill=idealcolor,tdplot_rotated_coords] (4,2,4) circle (2 pt);
	\draw[black,fill=idealcolor,tdplot_rotated_coords] (4,3,4) circle (2 pt);
	\draw[black,fill=idealcolor,tdplot_rotated_coords] (4,4,4) circle (2 pt);

	\end{tikzpicture}
	 \caption{$Q = \N \left[\begin{smallmatrix}0&1&0&1\\0&0&1&1\\ 1&1&1&1 \end{smallmatrix}\right]$ and $\deg(\midI)=\left\< \left[\begin{smallmatrix} 2& 2 & 2 & 3 \\ 0 & 1 & 3 & 3 \\ 2 & 2 & 3 & 3 \end{smallmatrix}\right]\right\>$}
	 \label{fig:3d_stdpairs}
    \end{subfigure}
\caption{Degree pairs of $\field[\affinesemigp]/\midI$}
\label{fig:stdpairs}
\end{figure*}

We assert that localization by a monomial prime ideal $P_{\facF}$
generates an injective map between sets of overlap classes of degree
pairs, $\degpo(\module)$ and $\degpo(\module_{P_{\facF}})$. For
economy of notation, let $\module_{\facF}$ be the localization of
$\module$ by a monomial prime ideal $P_{\facF}$ corresponding to a
face $\facF \in \facelattice(\affinesemigp)$. 

To begin, we see that the face lattice
$\facelattice(\affinesemigp-\N\facF)$ of the affine semigroup
$\affinesemigp-\N\facF$ arising from localization by the face $\facF$ can be
identified as a subset of the face lattice
$\facelattice(\affinesemigp)$ as follows. 

\begin{lemma}
\label{lem:localizationFaceBijection}
The maps
\begin{align*}
\facelattice(\affinesemigp-\N\facF) &\to \{ \facG \in \facelattice(\affinesemigp)\mid \facG \supset \facF \} &\text{ given by } & &\N\facG' &\mapsto \N\left(\facG' \cap \genset\right) \\
\{ \facG \in \facelattice(\affinesemigp)\mid \facG \supset \facF \} &\to \facelattice(\affinesemigp-\N\facF)  &\text{ given by } & &\N\facG &\mapsto \N\facG -\N\facF
\end{align*}
are bijective.
\end{lemma}

\begin{proof}
It suffices to show that $\facelattice(\affinesemigp -
\N\facF)$ and  $\{\facG \in \facelattice(\affinesemigp)\mid
\facG \supseteq \facF\}$ are in bijection.
Fix $\facG \in
\facelattice(\affinesemigp)$. Let $\vecc$ be an outer normal vector so
that ${\facG} =
\text{face}_{\vecc}({\affinesemigp})$. Recall that the absolute maximum of the functional $\<
\vecc,-\>$ on $\realize{\affinesemigp}$ is zero since every face of
$\affinesemigp$ contains the origin. 

We claim
\[
\text{face}_{\vecc}(\affinesemigp-\N\facF) = \begin{cases}
  \varnothing & \text{ if } \facF\not\subseteq \facG \\
  \facG\cup(-\facF) & \text{ if } \facF \subseteq
  \facG.  \end{cases}
\]

If $\facF$ is not a face of $\facG$, there exists a nonzero element $\vecf \in \facF \minus \facG$ such that $\< \vecc,\vecf\> < \< \vecc,\vecg \>=0$ for any $\vecg \in \facG$. Since $\< \vecc,-\smallm\vecf\>$ diverges when $\smallm \to \infty$, $\text{face}_{\vecc}(\affinesemigp-\N\facF)=\varnothing$. If $\facF$ is a face of $\facG$, $\facG\cup(-\facF) \subseteq face_{\vecc}(\affinesemigp-\N\facF)$. Pick $\vecb \in face_{\vecc}(\affinesemigp-\N\facF) \cap (\affinesemigp-\N\facF)$. Then, $\vecb = \veca - \vecf$ for some $\veca \in \affinesemigp$ and $\vecf \in \N\facF$. Since $0=\< \vecc,\vecb\> = \< \vecc,\veca\>$, $\veca \in \N\facG$. Thus, $\vecb\in\facG\cup(-\facF)$. 
\end{proof}

As a consequence of this result, we can express any face of
$\affinesemigp-\N\facF$ as $\N\facG -\N\facF$ for some face $\facG \in
\facelattice(\affinesemigp)$ such that $\facG \supset
\facF$. Likewise, $(\veca,\facG \cup (-\facF))$ denotes a degree pair
of a $\field[\affinesemigp-\N\facF]$-module $\module_{\facF}$.

Our next step is to 
show that each degree pair of a localization of $\module$ can be
lifted to a degree pair of $\module$. 

\begin{lemma}
\label{lem:stdPairEquivalentDefinition}
Suppose $\facG \supseteq \facF \in \facelattice(\affinesemigp)$. Given a degree pair $(\veca,\facG \cup(-\facF))$ of $\module_{\facF}$, there exists $\veca' \in \deg(\module)$ such that $(\veca', \facG \cup(-\facF)) = (\veca,\facG \cup(-\facF))$ and $(\veca',\facG)$ is a degree pair of $\module$.
\end{lemma}

By abuse of notation, let $\vart^{-\infty}=0 \in
\field[\affinesemigp]$.

\begin{proof}
Assume that $\{m_{1},m_{2},\cdots, m_{l}\}$ is a minimal generating set of $\module$ with $\deg(m_{i})=\veca_{i}$. Select an appropriate $\vecf \in \N\facF$ so that $\veca+\vecf \in \affinesemigp$. Let $\vecc_{i}=\veca+\vecf-\veca_{i}$ if $\vecc_{i} \in \deg(\affinesemigp)$ and $\vart^{\vecc_{i}}m_{i} \neq 0$ or $\vecc_{i}=-\infty$ otherwise. Set $m:= \left(\sum_{i=1}^{l}\vart^{\vecc_{i}}m_{i}\right) \in \module$. Then, $m/\vart^{\vecf} \in \module_{\facF}$ is a nonzero homogeneous element of degree $\veca$; otherwise no element of $\module_{\facF}$ with degree $\veca$ can be generated. Hence, $m$ is a nonzero homogeneous element of degree $\veca+\vecf$. Also, $(\deg(m),\facG)$ is a proper pair of $\module$, otherwise, if $\veca+\vecf+\vecg \not\in \deg(\module)$, then no element of $\module_{\facF}$ with degree $\veca+\vecf+\vecg$ exists. Thus, a degree pair $(\veca',\facG')$ exists that contains $(\deg(m),\facG)$ with $\facG' \supseteq \facG$. We may assume that $m'=\sum_{i=1}^{l}\vart^{\vecc_{i}'}m_{i}$ is of order $\veca'$ with $\vecc_{i} = \veca'-\veca_{i} \in \affinesemigp$ or $\vecc_{i}=-\infty$. Since $\deg(m) = \veca' +\vecg'$ for some $\vecg' \in \N\facG'$ and $\vecc_{i}' \neq -\infty$ if $\vecc_{i} \neq -\infty$, $\vart^{\vecg}m' = m$.

Furthermore, we claim $\facG' = \facG$. Suppose not, then we can have $\vecc \in \N\facG' \smallsetminus \N\facG$ such that $\veca+\vecc \not\in \deg(\module_{\facF})$ by the maximality of $(\veca,\facG \cup(-\facF))$. Thus, $\vart^{\vecc} \cdot m/\vart^{\vecf}=0$, which implies $\vart^{\vecc+\vecg'} \cdot m'=0$, contradicting the fact that $\vecc+\vecg'+\veca' \in \deg(\module)$. Hence, $\vecg' \in \N\facG$.

Finally, we assert $(\veca',\facG\cup(-\facF)) =
(\veca,\facG\cup(-\facF))$. Indeed $\veca= \veca' +\vecg'-\vecf$
indicates that $\veca \in \veca'+\N\left(\facG\cup(-\facF)\right)$,
implying $(\veca',\facG\cup(-\facF)) >
(\veca,\facG\cup(-\facF))$. Also, $(\veca',\facG\cup(-\facF))$ is a
proper pair; otherwise, we would not have an element whose degree is
in $\veca+\N\left(\facG\cup(-\facF)\right)$, contradiction. These two
degree pairs are same due to the maximality of $(\veca,\facG\cup(-\facF))$.
\end{proof}

The choice of $\veca'$ is not unique;
see~\cref{ex:std_localization}(\ref{enum:3d_localization}). Fortunately,
their overlap class is uniquely determined.

\begin{lemma}
\label{lem:lifting_of_overlap_is_unique}
Suppose $\facG \supseteq \facF \in \facelattice(\affinesemigp)$. Given
two overlapping degree pairs $(\veca,\facG \cup(-\facF))$ and
$(\vecb,\facG \cup(-\facF))$ of $\module_{\facF}$, let $\veca'$ and
$\vecb'$ be degrees of $\deg(M)$ chosen by~\cref{lem:stdPairEquivalentDefinition}. Then, $(\veca',\facG)$ and $(\vecb',\facG)$ overlap.
\end{lemma}

\begin{proof}
From $(\veca,\facG \cup(-\facF))= (\veca',\facG \cup(-\facF))$ and
$(\vecb,\facG \cup(-\facF)) = (\vecb',\facG \cup(-\facF))$, there
exists $\vecg_{a},\vecg_{b} \in \N\facG, \vecf_{a},\vecf_{b} \in
\N\facF$ such that $\veca'+\vecg_{a}-\vecf_{a} =
\veca+\vecg_{b}-\vecf_{b}$. Hence, $\veca'+\vecg_{a}+\vecf_{b}=
\veca+\vecg_{b}+\vecf_{a} \in \deg(\module_{\facF})$. Again,
$\veca'+\vecg_{a}+\vecf_{b} \in \deg(\module)$ when a set of minimal
generators is fixed and a similar construction in the proof
of~\cref{lem:stdPairEquivalentDefinition} is used.  
\end{proof}

As a consequence of the lemma above, we obtain the desired injective map
between sets of overlap classes under localization. We provide a new notation to describe this map. Given an overlap class $[\veca,\facF] \in \degpo(\module)$, let $\bigcup[\veca,\facF]:=\bigcup_{(\vecb,\facF) \in [\veca,\facF]}\vecb + \N\facF$. In other words, $\bigcup[\veca,\facF]$ is the union of all translates of faces represented by degree pairs in $[\veca,\facF]$.

\begin{theorem}
\label{thm:map_between_sets_of_equiv_classes_of_degree_pairs}
Suppose $\facF \subseteq \facG \subseteq \facH$ are faces of $\affinesemigp$. Given an overlap class $\left[\veca,\facH \cup (-\facG)\right] \in \degpo(\module_{\facG})$, there exists a unique overlap class $\left[\veca', \facH \cup (-\facF)\right] \in\degpo(\module_{\facF})$ such that
\[
\bigcup[\veca',\facH\cup(-\facF)] =\left(\bigcup[\veca,\facH\cup(-\facG)]\right)\cap \deg(\module_{\facF}).
\]
We denote this injection $\left[\veca', \facH \cup (-\facF)\right] =
\rest_{\facG, \facF}\left(\left[\veca,\facH \cup
    (-\facG)\right]\right)$, and call it \emph{restriction}. These
restriction maps satisfy that for any $\facF \subseteq \facG,\facG'
\subseteq \facH$, $\rest_{\facH,\facG} \circ\rest_{\facG,\facF} =
\rest_{\facH,\facF}= \rest_{\facH,\facG'} \circ\rest_{\facG',\facF}$.  
\end{theorem}

\begin{proof}[Proof of~\cref{thm:map_between_sets_of_equiv_classes_of_degree_pairs}]
It is sufficient to show that the map is defined in the case
$\facF=0$. For a given overlap class $[\veca,\facH] \in
\degpo(\module_{\facG})$, let $\rest_{\facG,0}([\veca,\facH])
:=[\veca',\facH]$ where $\veca'$ is determined
by~\cref{lem:stdPairEquivalentDefinition}. As demonstrated
in~\cref{lem:lifting_of_overlap_is_unique}, $\rest_{\facG,0}$ is
well-defined and injective. Moreover, by analogy to the construction
of elements of $\module_{\facF}$ with degrees in
$\bigcup[\veca,\facH]$ in the proof
of~\cref{lem:stdPairEquivalentDefinition}, $\bigcup[\veca,\facH] =
\left(\bigcup[\veca',\facH]\right) \cap \deg(\module)$. Associativity
is clear from the definition.
\end{proof}

Finally, we provide a statement about void pairs, which is used in~\cref{sec:Cohen--Macaulay}.

\begin{corollary}
  \label{cor:for_reproving_trung_hoa}
Let
$\{\facF_{i}\}_{i=1}^{m}$ be the set of all facets of a
(not-necessarily pointed) affine semigroup $\affinesemigp$.
Let
$\module:=\field[\affinesemigp_{\text{sat}}]/\field[\affinesemigp]$. If
$\affinesemigp\neq\bigcap_{i=1}^{m}\affinesemigp-\N\facF_{i}$, then
there exists a void pair $(\veca,\facF)$ such that $\facF$ is not a
facet. 
\end{corollary}

\begin{proof}
By the previous results, if the only void pairs $\affinesemigp$ arise
from facets, then $\affinesemigp
=\bigcap_{i=1}^{m}\affinesemigp-\N\facF_{i}$.
\end{proof}

\begin{example}[Continuation of \cref{ex:stdpairs}]
\label{ex:std_localization}
\
\begin{enumerate}[leftmargin=*]
\item \label{enum:frob_localization}
Given $\module=\field\left[\N\left[\begin{smallmatrix} 1 & 1 & 1 & 1
      \\ 0 & 1 & 3 &
      4\end{smallmatrix}\right]\right]/\left\<\vart^{\left[\begin{smallmatrix}
        1 \\ 1\end{smallmatrix}\right]}\right\>$, all overlap classes
of $\deg(M)$ are singletons. Indeed,

\begin{align*}
\degp(\module)&=\left\{\text{(green) }\textcolor{green!80!black}{\left( \left[\begin{smallmatrix} 2\\ 3\end{smallmatrix}\right], \varnothing \right)},\text{(blue) }\textcolor{blue}{\left( \left[\begin{smallmatrix} 0\\ 0\end{smallmatrix}\right], F_{1} \right)}, \text{(red) }\textcolor{red}{\left( \left[\begin{smallmatrix} 0\\ 0\end{smallmatrix}\right], F_{2} \right),\left( \left[\begin{smallmatrix} 1\\ 3\end{smallmatrix}\right], F_{2} \right), \left( \left[\begin{smallmatrix} 2\\ 6\end{smallmatrix}\right], F_{2} \right)} \right\}\\
\degp(\module_{\facF_{1}})&=\left\{ \text{(blue) }\textcolor{blue}{\left( \left[\begin{smallmatrix} 0\\ 0\end{smallmatrix}\right], \facF_{1} \cup (-\facF_{1}) \right)}\right\}\\
\degp(\module_{\facF_{2}})&=\left\{ \text{(red) }\textcolor{red}{\left( \left[\begin{smallmatrix} 0\\ 0\end{smallmatrix}\right], \facF_{2}\cup(-\facF_{2}) \right),\left( \left[\begin{smallmatrix} 1\\ 3\end{smallmatrix}\right], \facF_{2}\cup(-\facF_{2}) \right), \left( \left[\begin{smallmatrix} 2\\ 6\end{smallmatrix}\right], \facF_{2}\cup(-\facF_{2}) \right)}\right\}.
\end{align*}

This shows two injections $\degpo(\midI_{\facF_{1}})\hookrightarrow\degpo(\midI) \hookleftarrow \degpo(\midI_{\facF_{2}})$.
\item \label{enum:3d_localization}
Given $\module=\field\left[\N \left[\begin{smallmatrix}0 & 1 & 1 & 0 \\ 0 & 0 & 1 & 1 \\ 1 & 1 & 1 & 1 \end{smallmatrix}\right]\right]/\left\<\vart^{\left[\begin{smallmatrix} 2& 2 & 2 & 3 \\ 0 & 1 & 3 & 3 \\ 2 & 2 & 3 & 3 \end{smallmatrix}\right]}\right\>$, the set of degree pairs of ideals in each localizations are as follows. 
\begin{align*}
\degp(\module)&=\left\{ \text{(green) }\textcolor{green!80!black}{\left( \left[\begin{smallmatrix} 2\\ 2 \\ 2\end{smallmatrix}\right], \varnothing \right)}, \text{(blue) } \textcolor{blue}{\left( \left[\begin{smallmatrix} 0\\ 0 \\ 0\end{smallmatrix}\right], \facF_{4} \right)}, \text{(yellow) }\textcolor{yellow!80!black}{\left( \left[\begin{smallmatrix} 1\\ 0 \\ 1\end{smallmatrix}\right], \facF_{4} \right)}, \text{(red) }\textcolor{red}{\left( \left[\begin{smallmatrix} 1\\ 1 \\ 1\end{smallmatrix}\right], \facF_{4} \right)}\right\}\\
\degp(\module_{\facG})&=\left\{ \text{(blue) } \textcolor{blue}{\left( \left[\begin{smallmatrix} 0\\ 0 \\ 0\end{smallmatrix}\right], \facF_{4} \cup (-\facG) \right)}, \text{(orange) }\textcolor{orange}{\left( \left[\begin{smallmatrix} 1\\ 0 \\ 0\end{smallmatrix}\right], F_{4} \cup(-\facG) \right)} \right\}
\end{align*}
for any $\facG \in \{\veca_{1},\veca_{4},\facF_{4} \}.$ Indeed, 
\[
\left( \left[\begin{smallmatrix} 1\\ 0 \\ 0\end{smallmatrix}\right], F_{4} \cup(-\facF_{4}) \right)=\left( \left[\begin{smallmatrix} 1\\ 0 \\ 1\end{smallmatrix}\right], F_{4} \cup(-\facF_{4}) \right) =\left( \left[\begin{smallmatrix} 1\\ 1 \\ 1\end{smallmatrix}\right], F_{4} \cup(-\facF_{4}) \right).
\]
This is an example of how the selection of $\veca'$ in~\cref{lem:stdPairEquivalentDefinition} is not unique. Nonetheless, we have an injection $\degpo(\midI_{\facG}) \to \degpo(\midI)$ by sending the orange overlap class to the overlap class consisting of yellow and red standard pairs.
\end{enumerate}
\end{example}

\section{Degree space with degree pair topology}
\label{sec:stdpair_top}

We now construct the \emph{degree space} of $\module$,
$\degs(\module):=\bigcup_{\facF \in
  \facelattice(\affinesemigp)}\deg(\module_{\facF})$, a topological
space formed by gluing all degree pairs of
localizations of a module. This structure enables us to simultaneously record
all degrees resulting from localizations. Moreover, the minimal open
sets of $\degs(\module)$, called \emph{grains}, partition all degrees
belonging to a fixed collection of localizations. A grain's \emph{chaff} is the poset of
all localizations that contain the given grain. Together
with~\cref{subsec:Ishida}, these tools yield a Hochster-type
formula for the Hilbert series of the local cohomology of $\module$
in~\cref{sec:Hochster}.  

\begin{definition}[Degree space and degree pair topology]
\label{def:degp_topology}
The \emph{degree space} $\degs(\module)$ of a finitely generated $\field[\affinesemigp]$-module $\module$ is the union of $\deg(\module_{\facF})$ for all faces $\facF$ of $\facelattice(\affinesemigp)$. The \emph{degree pair topology} is the smallest topology on $\degs(\module)$ such that for any face $\facF \in \facelattice(\affinesemigp)$ and for an overlap class $[\veca,\facG \cup (-\facF)] \in \degpo(\module_{\facF})$, the set $\bigcup[\veca,\facG \cup (-\facF)]$ is both open and closed.
\end{definition}

\begin{definition}[Grain and chaff]
\label{def:grain_and_chaff}
In the degree pair topology, we refer to a minimal nonempty open set
as a \emph{grain} of $\degs(\module)$. Let $\grainset(\module)$ be the
set of all such grains. The \emph{chaff} of a grain $\grain$, $D_{\grain}$, is defined as the collection of all localizations of $\module$ containing $\grain$. 
\end{definition}

These names are inspired by the agricultural metaphors of
Grothendieck; we bundle degree pairs on $\degs(\module)$ and thresh
(topologize) them in order to obtain grains. Chaff is a layer of grain
that provides information about the grain's containment in certain
localizations. 

\begin{remark}
  The sectors and sector partition introduced in~\cite{MR2199183}
  are almost the same the grains and chaff used in this
  article. The main differences are the topological context, and that
  grains actually refine the sector
  partition.
\end{remark}

\begin{lemma}
\label{lem:finiteness_of_degp_top}
The degree pair topology has finitely many open sets. 
\end{lemma}
\begin{proof}
From~\cref{lem:finiteness_of_degree_pairs}, $\degpo(\module_{\facF})$ is finite. Also, $\facelattice(\affinesemigp)$ is finite. Finally, a subbase including all overlap classes and their complements over all localizations is used to generate the topology. Hence the topology has finitely many open sets.
\end{proof}

Our next result is that the grain set $\grainset(\affinesemigp/\idI)$
partitions the degrees of a module.

\begin{proposition}
\label{pro:grain_partition}
$\grainset(\module)$ partitions $\degs(\module)$ and is therefore a basis of the degree pair topology.
\end{proposition}

\begin{proof}
It suffices to show that $\grainset(\module)$ partitions
$\degs(\module)$. First of all, for any two elements $\setS$ and
$\setS'$ of $\grainset(\module)$, $\setS \cap
\setS'=\varnothing$. Otherwise,  $\setS$ and $\setS'$ cannot be
minimal nonempty opens, a contradiction. To see that $\grainset(\module)$ covers $\degs(\module)$, suppose $\veca \in \deg(\module_{\facF})$ for some face $\facF$. Let 
\[
\setS := \left( \bigcap_{\substack{\veca \in [\vecb,\facG] \in
      \degpo(\module_{\facF'}) \\ \facF' \in
      \facelattice(\affinesemigp) }} \left(\bigcup[\vecb,\facG]
  \right)\right) \cap \left( \bigcap_{\substack{\veca \not\in
      [\vecb,\facG] \in \degpo(\module_{\facF'}) \\ \facF' \in
      \facelattice(\affinesemigp) }}
  \left(\bigcup[\vecb,\facG]\right)^{c} \right).
\]
This is a nonempty open set since $\veca \in \setS$. We claim that
$\setS \in \grainset(\module)$. Suppose not; then there exists an
open set $\setS' \subsetneq \setS$. By the property of the subbase, we
may let $\setS' \subseteq \setS \cap \left(\bigcup[\vecb,\facG]\right)
\subsetneq \setS$ or $\setS' \subseteq \setS \cap
\left(\bigcup[\vecb,\facG]\right)^{c}\subsetneq \setS$ for some
overlap class $[\vecb,\facG]$. If $\setS' \subseteq \setS
\cap\left(\bigcup[\vecb,\facG]\right) \subsetneq \setS$ holds, then
$\left(\bigcup[\vecb,\facG]\right)^{c}$ contains $\veca$, thus
$\left(\bigcup[\vecb,\facG]\right)^{c} \cap \setS =\setS$ by the
construction of $\setS$. This implies that $\setS'=\varnothing$. If
$\setS' \subseteq \setS \cap
\left(\bigcup[\vecb,\facG]\right)^{c}\subsetneq \setS$ holds, then
$\setS \cap \left(\bigcup[\vecb,\facG]\right) = \setS$ implies
$\setS'=\varnothing$. In both cases, $\setS'$ is empty, a
contradiction. 
\end{proof}

\begin{figure*}[t!]
\centering
\begin{subfigure}[c]{0.45\linewidth}
\centering
\begin{tikzpicture}[scale=0.6]
  \draw[step=1cm,gray,very thin] (-3,-6) grid (3,6);
  \draw [<->] (0,6.5) node (yaxis) [above] {$y$}
        |- (3.5,0) node (xaxis) [right] {$x,\facF_{1}$};
\node[above] at (1.625,6.5) {$\facF_{2}$};        
\node[below left] at (0,0) {$O$};
\draw(0,-6.5) -- (0,0);
\draw(-3.5,0) -- (0,0);

\draw[black,fill=stdmcolor] (0,0) circle (3pt);

\draw[black,fill=stdmcolor] (2,0) circle (3pt);
\draw[black,fill=stdmcolor] (3,0) circle (3pt);
\draw[black,fill=stdmcolor] (3,0) circle (3pt);
\draw[black,fill=stdmcolor] (2,6) circle (3pt);

\foreach \y in {0,3,4}
\draw[black,fill=stdmcolor] (1,\y) circle (3pt);

\draw[black,fill=stdmcolor] (2,3) circle (3pt);
\draw[black,fill=stdmcolor] (1,2) circle (3pt);

\draw[black,fill=stdmcolor] (0,-1) circle (3pt);
\draw[black,fill=stdmcolor] (0,-2) circle (3pt);
\draw[black,fill=stdmcolor] (-1,-4) circle (3pt);
\draw[black,fill=stdmcolor] (-1,-5) circle (3pt);
\draw[black,fill=stdmcolor] (-1,-6) circle (3pt);
\draw[black,fill=stdmcolor] (-1,0) circle (3pt);
\draw[black,fill=stdmcolor] (-2,0) circle (3pt);
\draw[black,fill=stdmcolor] (-3,0) circle (3pt);

\draw[red,ultra thick,->] (1,4) -- (1.625,6.5);
\draw[red,ultra thick,->] (1,3) -- (1.875,6.5);
\draw[red,ultra thick,->] (2,6) -- (2.125,6.5);
\draw[orange,ultra thick,->] (1,2) -- (-1,-6);
\draw[orange,ultra thick,->] (0,-1) -- (-1.25,-6);
\draw[orange,ultra thick,->] (-1,-4) -- (-1.5,-6);
\draw[blue,ultra thick,->] (1,0) -- (3.5,0);
\draw[cyan,ultra thick,->] (-1,0) -- (-3.5,0);
\draw[green!80!black,ultra thick] (2,3) circle (10 pt);
\draw[yellow!80!black,ultra thick] (0,0) circle (10 pt);

\end{tikzpicture}
\caption{$Q = \N \left[\begin{smallmatrix}1 & 1 & 1 & 1 \\ 0 & 1 & 3 & 4  \end{smallmatrix}\right]$ and $I=\left\< \left[\begin{smallmatrix} 1 \\ 1 \end{smallmatrix}\right]\right\>$}
\label{fig:frob_stdm}
\end{subfigure}
\begin{subfigure}[c]{0.45\linewidth}
\centering
\begin{tikzpicture}[scale=0.6]
  \draw[step=1cm,gray,very thin] (-3,-3) grid (3,2);
  \draw [<->] (0,2.5) node (yaxis) [above] {$z$}
        |- (3.5,0) node (xaxis) [right] {$y$};
\draw (0,0) -- (-3.5,0);
\draw (0,0) -- (0,-3.5);
\foreach \x in {-3,...,3}
\foreach \y in {-3,...,2}
\draw[black,fill=stdmcolor] (\x,\y) circle (3pt);
\node[below left] at (0,0) {$O$};
\draw[red!80!white,ultra thick,->] (0,0) -- (0,2.5);
\draw[red!80!white,ultra thick,->] (0,0) -- (2.5,2.5);
\fill[red!80!white,opacity =0.5] (0,2) -- (0,0) -- (2,2);
\draw[blue,ultra thick,->,opacity =0.7] (0,-1) -- (0,-3.5);
\draw[blue,ultra thick,->,opacity =0.7] (0,-1) -- (3,2);
\fill[blue, opacity=0.5] (0,-3) -- (0,-1) -- (3,2)  -- (3,-3) -- cycle;
\draw[cyan,ultra thick,->,opacity =0.7] (-1,-1) -- (-3.5,-3.5);
\draw[cyan,ultra thick,->,opacity =0.7] (-1,-1) -- (-1,2.5);
\fill[cyan,opacity =0.5] (-3,-3) -- (-3,2) -- (-1,2) -- (-1,-1) -- cycle;
\draw[orange,ultra thick,->,opacity =0.7] (-1,-2) -- (-2.5,-3.5);
\draw[orange,ultra thick,->,opacity =0.7] (-1,-2) -- (-1,-3.5);
\fill[orange,opacity =0.7] (-2,-3) -- (-1,-2) -- (-1,-3);
\end{tikzpicture}
\begin{tikzpicture}[scale=0.6]
  \draw[step=1cm,gray,very thin] (-3,-3) grid (3,2);
  \draw [<->] (0,2.5) node (yaxis) [above] {$z$}
        |- (3.5,0) node (xaxis) [right] {$y$};
\draw (0,0) -- (-3.5,0);
\draw (0,0) -- (0,-3.5);
\foreach \x in {-3,...,3}
\foreach \y in {-3,...,2}
\draw[black,fill=stdmcolor] (\x,\y) circle (3pt);
\draw[violet,ultra thick] (0,0) circle (10 pt);
\draw[red!80!white,ultra thick,->] (1,1) -- (0,1) -- (0,2.5);
\draw[red!80!white,ultra thick,->] (1,1) -- (2.5,2.5);
\fill[red!80!white,opacity =0.5]  (0,2)--(0,1) -- (1,1) -- (2,2);
\draw[blue,ultra thick,->,opacity =0.7] (0,-1) -- (0,-3.5);
\draw[blue,ultra thick,->,opacity =0.7] (0,-1) -- (3,2);
\fill[blue, opacity=0.5] (0,-3) -- (0,-1) -- (3,2)  -- (3,-3) -- cycle;
\draw[cyan,ultra thick,->,opacity =0.7] (-1,-1) -- (-3.5,-3.5);
\draw[cyan,ultra thick,->,opacity =0.7] (-1,-1) -- (-1,2.5);
\fill[cyan,opacity =0.5] (-3,-3) -- (-3,2) -- (-1,2) -- (-1,-1) -- cycle;
\draw[orange,ultra thick,->,opacity =0.7] (-1,-2) -- (-2.5,-3.5);
\draw[orange,ultra thick,->,opacity =0.7] (-1,-2) -- (-1,-3.5);
\fill[orange,opacity =0.7] (-2,-3) -- (-1,-2) -- (-1,-3);
\node[below left] at (0,0) {$(1,0,0)$};
\end{tikzpicture}
\caption{$Q = \N \left[\begin{smallmatrix}0&1&1&0\\0&0&1&1\\ 1&1&1&1 \end{smallmatrix}\right]$ and $I=\left\< \left[\begin{smallmatrix} 2& 2 & 2 & 3 \\ 0 & 1 & 3 & 3 \\ 2 & 2 & 3 & 3 \end{smallmatrix}\right]\right\>$}
\label{fig:3d_coin_stdm}
\end{subfigure}
\caption{$\degs(\affinesemigp/\idI)$ with grains}
\label{fig:topology}
\end{figure*}

\begin{example}[Continuation of \cref{ex:std_localization}]
\label{ex:topology}
\
\begin{enumerate}[leftmargin=*]
\item \label{enum:frob_topology}
Given $\module=\field\left[\N\left[\begin{smallmatrix} 1 & 1 & 1 & 1 \\ 0 & 1 & 3 & 4\end{smallmatrix}\right]\right]/\left\<xy\right\>$, $\degs(\module)$ is the union of integral points in $y=4x, y=4x-1, y=4x-2, y=0$ and $\{(2,2)\}$. Moreover, 10 grains
\begin{align*}
&\text{(red) }\textcolor{red}{\left[\begin{smallmatrix} 1 \\ 4\end{smallmatrix}\right]+\N\facF_{2}, \left[\begin{smallmatrix} 1 \\ 3\end{smallmatrix}\right]+\N\facF_{2},\left[\begin{smallmatrix} 2 \\ 6\end{smallmatrix}\right]+\N\facF_{1}}, \text{(blue) }\textcolor{blue}{\left[\begin{smallmatrix} 1 \\ 0\end{smallmatrix}\right]+\N\facF_{1},} \text{(cyan) } \textcolor{cyan}{\left[\begin{smallmatrix} -1 \\ 0\end{smallmatrix}\right]+ \N\left(-\facF_{1}\right),} \\
&\text{(orange) }\textcolor{orange}{\left[\begin{smallmatrix} -1 \\ 4\end{smallmatrix}\right]+ \N\left(-\facF_{2} \right),\left[\begin{smallmatrix} 0 \\ -1\end{smallmatrix}\right]+\N\left(-\facF_{2} \right),\left[\begin{smallmatrix} 1 \\ 2\end{smallmatrix}\right]+\N \left(-\facF_{2} \right),}
\text{(yellow) }\textcolor{yellow!80!black}{\left[\begin{smallmatrix} 0 \\ 0\end{smallmatrix}\right]},\text{(green) } \textcolor{green!80!black}{\left[\begin{smallmatrix} 2 \\ 3\end{smallmatrix}\right]}
\end{align*}
are depicted in~\cref{fig:frob_stdm}. Two grains with the same color have the same chaff. Indeed, for the given grain $\grain$ with color from~\cref{fig:frob_stdm},
\begin{align*}
\text{(red) }\textcolor{red}{D_{\grain}}&:=\{ 0,\facF_{2} \},&  \text{(blue) }\textcolor{blue}{D_{\grain}}&:=\{ 0,\facF_{1} \}, &\text{(cyan) }\textcolor{cyan}{D_{\grain}}&:=\{ \facF_{1} \}, \\
\text{(orange) }\textcolor{orange}{D_{\grain}}&:=\{ \facF_{2} \}, &\text{(green) }\textcolor{green!80!black}{D_{\grain}}&:=\{ 0 \},& \text{(yellow) }\textcolor{yellow!80!black}{D_{\grain}}&:=\{ 0,\facF_{1},\facF_{2} \}.
\end{align*}
Note that $\left[\begin{smallmatrix} 1 \\ 2\end{smallmatrix}\right]$
is a hole filled by the localization with respect to $\facF_{2}$, and
therefore lies only in the degree pair of $\module_{\facF_{2}}$. 
\item \label{enum:3d_topology}
Given $\module=\field\left[\N \left[\begin{smallmatrix}0 & 1 & 1 & 0
      \\ 0 & 0 & 1 & 1 \\ 1 & 1 & 1 &
      1 \end{smallmatrix}\right]\right]/\left\<x^{2}z^{2},x^{2}yz^{2},x^{2}y^{3}z^{3},x^{3}y^{3}z^{3}\right\>$,
$\degs(\module)$ consists of the  $yz$-plane $(x=0)$, its translation $x=1$, and the point $(2,2,2)$. 10 grains
\begin{align*}
&\text{(red) }\textcolor{red}{\left[\begin{smallmatrix} 0 \\ 0 \\ 0\end{smallmatrix}\right]+\N\facF_{4}, \left(\left\{\left[\begin{smallmatrix} 1 \\ 0 \\ 1\end{smallmatrix}\right],\left[\begin{smallmatrix} 1 \\ 1 \\ 1 \end{smallmatrix}\right]\right\}+\N\facF_{4}\right)},\text{(blue) }\textcolor{blue}{\left[\begin{smallmatrix} 0 \\ 0 \\-1\end{smallmatrix}\right]+\N\left[\begin{smallmatrix} -\veca_{1} \\ \veca_{4}\end{smallmatrix}\right]^{t},\left[\begin{smallmatrix} 1 \\ 0 \\-1\end{smallmatrix}\right]+\N\left[\begin{smallmatrix} -\veca_{1}\\ \veca_{4}\end{smallmatrix}\right]^{t}}, \text{(violet) }\textcolor{violet}{\left[\begin{smallmatrix} 1 \\ 0 \\0\end{smallmatrix}\right]} \\
&\text{(cyan) }\textcolor{cyan}{\left[\begin{smallmatrix} 0 \\ -1 \\-1\end{smallmatrix}\right]+\N\left[\begin{smallmatrix} \veca_{1} \\ -\veca_{4}\end{smallmatrix}\right]^{t},\left[\begin{smallmatrix} 1 \\ -1 \\-1\end{smallmatrix}\right]+\N\left[\begin{smallmatrix} \veca_{1} \\ -\veca_{4}\end{smallmatrix}\right]^{t}}, \text{(orange) }\textcolor{orange}{\left[\begin{smallmatrix} 0 \\ -1 \\-2\end{smallmatrix}\right]-\N\facF_{4},\left[\begin{smallmatrix} 1 \\ -1 \\-2\end{smallmatrix}\right]-\N\facF_{4}},\text{(green) }\textcolor{green!80!black}{\left[\begin{smallmatrix} 2 \\ 2 \\2 \end{smallmatrix}\right] }
\end{align*}
are depicted in $x=1$ and $x=0$ planes of~\cref{fig:3d_coin_stdm}
except $\left[\begin{smallmatrix} 2 \\ 2
    \\2 \end{smallmatrix}\right]$. Note that
$\left[\begin{smallmatrix} 1 \\ 0 \\0\end{smallmatrix}\right]$ is not
a monomial of $\affinesemigp$ but that of $\affinesemigp-\N\veca_{1}$
or $\affinesemigp-\N\veca_{4}$, and therefore it lies in the
intersection of two degree pairs that came from $\module_{\veca_{1}}$
and $\module_{\veca_{4}}$ respectively. For the given grain $\grain$
with color as in the figures, the chaffs are as follows.
\begin{align*}
\text{(red) }\textcolor{red}{D_{\grain}}&:=\{ 0,\veca_{1},\veca_{4},\facF_{4} \}&\text{(blue) }\textcolor{blue}{D_{\grain}}&:=\{ \veca_{1},\facF_{4} \}, &\text{(cyan) }\textcolor{cyan}{D_{\grain}}&:=\{ \veca_{4},\facF_{4}  \}, \\
\text{(orange) }\textcolor{orange}{D_{\grain}}&:=\{ \facF_{4} \}, &\text{(green) }\textcolor{green!80!black}{D_{\grain}}&:=\{ 0 \},&\text{(violet) }\textcolor{violet}{D_{\grain}}&:=\{ \veca_{1},\veca_{4},\facF_{4} \}.
\end{align*}
\end{enumerate}
\end{example}

\section{Hilbert series via grains}
\label{sec:Hochster}

We derive a Hochster-type formula for the Hilbert series of the local
cohomology of a finitely generated $\Z^{\dimd}$-graded
$\field[\affinesemigp]$-module $\module$. Quotients of affine semigroup
rings by monomial ideals are an important example. 

\begin{definition}[Degree $\veca$-piece of the transverse section and its reduced chain complex]
  \label{def:chain_cpx_of_T}
Let $\module$ be a finitely generated graded $\field[\affinesemigp]$-module.
Given an element $\veca \in \Z^{\dimd}$,
let $\module_{\veca}$ be the degree-$\veca$ graded piece of
$\module$. Let $\crosssection$ be the transverse section of $\affinesemigp$ from~\cref{subsec:Ishida}. The
degree-$\veca$ \emph{graded piece of $\crosssection$ over $\module$}
is the subset of the face lattice of $\crosssection$ given by

\[
\crosssection_{\veca}:= \{\facF \in \facelattice(\crosssection) \mid
\veca \in \deg(\module_{\facF}) \}.
\] 

Likewise, denote $\virtual{\crosssection_{\veca}}:=
\{\virtual{\facF}\mid \facF \in \crosssection_{\veca}\}$ the
degree-$\veca$ \emph{graded piece of the affine semigroup $\affinesemigp$ over
  $\module$}. To align with the Ishida complex, the homological degree
of the reduced chain
complex of $\crosssection_{\veca} $ must be shifted as follows:
\[
\begin{tikzcd}
\rchaincpx{\crosssection_{\veca}}: 0\arrow[r] & C^{d} \arrow[r,"\partial"] &C^{d-1} \arrow[r,"\partial"] & \cdots \arrow[r,"\partial"] & C^{0} \arrow[r,"\partial"] & 0, \quad C^{k} :=\bigoplus_{\substack{\facF \in \crosssection_{\veca} \\ \dim\facF=k-1}}\field \facF
\end{tikzcd}
\]
with the differential 
\[
\partial(\facF) := \sum_{\substack{\facG  \in \crosssection_{\veca} \\
    \dim\facG=k-1}}\incidence(\facF,\facG) \facG \text{ for }
\dim\facF=k,
\]
where $\incidence$ is an incidence function inherited from
$\crosssection$ and $\field\facF$ is a 1-dimensional $\field$-vector
space having $\{\facF\}$ as a basis. 
\end{definition}

\begin{lemma}
\label{lem:graded_part_is_crosssection}
$\rchaincpx{\crosssection_{\veca}}$ is well-defined and $\left(\ishida^{\bullet} \tensor{\field[\affinesemigp]}\module\right)=\Hom_{\field}(\rchaincpx{\crosssection_{\veca}}, \field).$
\end{lemma}

\begin{proof}
For any $k \in \N$,
\[
\left(\ishida^{k} \tensor{\field[\affinesemigp]}
  \module\right)_{\veca} = \left(\bigoplus_{\facF \in
    \facelattice(\crosssection)^{k-1}}
  \module_{\virtual{\facF}}\right)_{\veca} =
\bigoplus_{\substack{\facF \in \facelattice(\crosssection)^{k-1} \\
    \module_{\virtual{\facF}} \neq 0}}
\left(\module_{\virtual{\facF}}\right)_{\veca} \cong
\bigoplus_{\substack{\facF \in \facelattice(\crosssection)^{k-1} \\
    \veca \in \deg(\module_{\virtual{\facF}})  }} \field \facF
=\bigoplus_{\substack{\facF \in \crosssection_{\veca} \\
    \dim\facF=k-1}}\field \facF,
\]
which is equal to $C^{k}$. Apply the functor $\Hom_{\field}(-,\field)$
to obtain $\rchaincpx{\crosssection_{\veca}}$. Since differentials in
$\ishida^{\bullet} \tensor{\field[\affinesemigp]} \module$ are
$\field$-linear, their images under $\Hom_{\field}(-,\field)$ agree
with differentials in $\rchaincpx{\crosssection_{\veca}}$. 
\end{proof}

Furhtermore, if two elements of $\Z^{\dimd}$ are in the same grain, their graded pieces of the Ishida complex coincide.

\begin{lemma}
\label{lem:classify_crosssection}
For any $\veca \in \grain \in \grainset\left(\module\right)$,
$\virtual{\crosssection_{\veca}}= D_{\grain}$. Thus,
$\rchaincpx{\crosssection_{\veca}}= \rchaincpx{\crosssection_{\vecb}}$
if $\veca,\vecb \in \grain$. If there is no grain containing $\veca$,
then $\rchaincpx{\crosssection_{\veca}}=0$. 
\end{lemma}

\begin{proof}
By~\cref{pro:grain_partition}, if there is no grain containing
$\veca$, then $\veca \not\in \deg(\module_{\facF})$ for any $\facF \in
\facelattice(\affinesemigp)$, so
$\rchaincpx{\crosssection_{\veca}}=0$. $\virtual{\crosssection_{\veca}}=
D_{\grain}$ is clear from the definition of chaff. 
\end{proof}

As a consequence of the previous result, we may use the notation $\rchaincpx{\crosssection_{\grain}}:= \rchaincpx{\crosssection_{\veca}}$ for the grain $\grain$ containing $\veca$. $\rchaincpx{\crosssection_{\grain}}$ coincides with the chain complex of $\crosssection_{\grain}=D_{\grain}$. Since $\grainset(\module)$ is finite, according to~\cref{lem:finiteness_of_degp_top}, the Hilbert series of the local cohomology of $\module$ is a finite sum over cohomologies of chaffs as follows.

\begin{theorem}[Hochster-type formula for the Ishida complex]
\label{thm:Hochster}
The multi-graded Hilbert series for the local cohomology of a graded
module $\module$ with support at the maximal ideal $\maxid$ is
\[
\hilb{\localcoho{\maxid}{i}{\module}}{\hilbt} = \sum_{\grain
  \in\grainset(\module) }
\dim_{\field}H^{i}(\Hom_{\field}(\rchaincpx{\crosssection_{\grain}},\field)
) \sum_{\veca  \in \grain }\hilbt^{\veca}.
\]
\end{theorem}

\begin{proof}
\begin{align*}
\hilb{\localcoho{\maxid}{i}{\module}}{\hilbt} &= \sum_{\veca \in \Z^{\dimd}}\dim_{\field}\left(\localcoho{\maxid}{i}{\module}\right)_{\veca}\hilbt^{\veca} = \sum_{\veca  \in \Z^{\dimd}} \dim_{\field}\left(H^{i}(\ishida^{\bullet} \tensor{\field[\affinesemigp]} \module)\right)_{\veca} \hilbt^{\veca} \\
&=\sum_{\veca  \in \degs(\module)} \dim_{\field}H^{i}(\Hom_{\field}(\rchaincpx{\crosssection_{\veca}},\field) ) \hilbt^{\veca} \quad\quad\;\,\; \text{(\cref{lem:graded_part_is_crosssection})} \\
& =\sum_{\grain \in\grainset(\module) }\sum_{\veca  \in \grain } \dim_{\field}H^{i}(\Hom_{\field}(\rchaincpx{\crosssection_{\grain}},\field) ) \hilbt^{\veca} \quad \text{(\cref{lem:classify_crosssection})}\\
&=\sum_{\grain \in\grainset(\module) } \dim_{\field}H^{i}(\Hom_{\field}(\rchaincpx{\crosssection_{\grain}},\field) ) \left(\sum_{\veca  \in \grain }\hilbt^{\veca}\right)
\end{align*}
\end{proof}

The Hilbert series in Theorem~\ref{thm:Hochster} is a finite sum
involving generating functions of lattice points in polyhedra. To conclude these
generating functions are rational, the underlying cone must be
pointed~\cites{Barvinok99,Barvinok03}.

\begin{corollary}
\label{cor:Barvinok_algorithm}
If $\affinesemigp$ is pointed, the Hilbert series in
Theorem~\ref{thm:Hochster} can be expressed as a (formal) sum of
rational functions.  \qed
\end{corollary}

In the non-pointed case, Hochster-type formulas are
not necessarily given by rational functions. For example, the Hilbert series of the Laurent
polynomial ring $\field[\varx,\varx^{-1}]$ cannot be expressed as a
rational function, since $\frac{1}{1-\varx} +
\frac{\varx^{-1}}{1-\varx^{-1}} = 0$, which is different from the
formal sum $\sum_{i \in \Z}\varx^{i}$.  

Vanishing of local cohomology is a standard way to detect whether a ring is Cohen--Macaulay.

\begin{theorem}[Combinatorial Cohen--Macaulay criterion]
\label{thm:Cohen--Macaulayness}
Given a pointed affine semigroup ring $\field[\affinesemigp]$ and a
monomial ideal $\midI$, $\field[\affinesemigp]/\midI$ is
Cohen--Macaulay ring if and only if every chaff of grains in $
\grainset(\field[\affinesemigp]/\midI)$ is either acyclic or
(-1)-dimensional in homological index $\diml:= \dim
\field[\affinesemigp]/\midI$. \qed
\end{theorem}

\begin{example}[Continuation of \cref{ex:ishida}]
\label{ex:hochster}
\
\begin{enumerate}[leftmargin=*]
\item \label{enum:frob_hochster}
As illustrated in \cref{ex:topology}(\ref{enum:frob_topology}),
$\affinesemigp/I = \N\left[\begin{smallmatrix} 1 & 1 & 1 & 1 \\ 0 & 1
    & 3 & 4\end{smallmatrix}\right]/\left\<\left[\begin{smallmatrix} 1
      \\ 1\end{smallmatrix}\right]\right\>$ has six distinct colored
chaffs. We take the unions of grains of the same color and color-code these unions. The following table summarizes their rational generating functions.
\begin{align*}
\text{(red) }\textcolor{red}{f_{\text{r}}}&:= \frac{xy^4 + xy^3 + x^2y^6}{1-xy^4} &\text{(blue) }\textcolor{blue}{f_{\text{b}}}&:=\frac{x}{1-x}, &\text{(cyan) } \textcolor{cyan}{f_{\text{c}}}&:= \frac{1}{x-1},\\
\text{(orange) } \textcolor{orange}{f_{\text{o}}}&:=\frac{1 + xy^{3}+x^2y^{6}}{xy^4-1},  &\text{(green) }\textcolor{green!80!black}{f_{\text{g}}}&:=x^2 y^3,  & \text{(yellow) }\textcolor{yellow!80!black}{f_{\text{y}}}&:=1
\end{align*}
where $x:=\hilbt^{\left[\begin{smallmatrix} 1 \\ 0\end{smallmatrix}\right]}$ and $y:=\hilbt^{\left[\begin{smallmatrix} 0 \\ 1\end{smallmatrix}\right]}$. Thus, $\rchaincpx{\crosssection_{G}}$ is a member of one of three chain complexes below.
\begin{align*}
\rchaincpx{\crosssection_{\textcolor{yellow!80!black}{y}}}: & 0 \to \field \to \field^2 \to 0 & \rchaincpx{\crosssection_{\textcolor{red}{r}}},\rchaincpx{\crosssection_{\textcolor{blue}{b}}}: & 0 \to \field \to \field \to 0 \\ 
\rchaincpx{\crosssection_{\textcolor{green!80!black}{g}}}:  & 0 \to \field \to 0 \to 0  & \rchaincpx{\crosssection_{\textcolor{cyan}{c}}},\rchaincpx{\crosssection_{\textcolor{orange}{o}}}: & 0 \to 0 \to \field\to 0
\end{align*}
As a result,
\begin{align*}
\hilb{\localcoho{\maxid}{0}{\setS}}{\{x,y\}} &=\textcolor{green!80!black}{f_{g}} & \hilb{\localcoho{\maxid}{1}{\setS}}{\{x,y\}}&=\textcolor{yellow!80!black}{f_{y}}+\textcolor{cyan}{f_{c}}+\textcolor{orange}{f_{o}}.
\end{align*}
\item\label{enum:3d_hochster} 
As illustrated in \cref{ex:topology}(\ref{enum:3d_topology}),
$\affinesemigp/I =\N \left[\begin{smallmatrix}0 & 1 & 1 & 0 \\ 0 & 0 &
    1 & 1 \\ 1 & 1 & 1 &
    1 \end{smallmatrix}\right]/\left\<\left[\begin{smallmatrix} 2& 2 &
      2 & 3 \\ 0 & 1 & 3 & 3 \\ 2 & 2 & 3 &
      3 \end{smallmatrix}\right]\right\>$ has six distinct colored
chaffs. As before, we take unions of grains of the same color and
index these unions according to their color. Their rational generating
functions are as follows:
\begin{align*}
\text{(red) }\textcolor{red}{f_{r}}&:= \frac{1+x}{(1-z)(1-yz)} -x,  &\text{(blue) }\textcolor{blue}{f_{b}}&:=\frac{1+x}{(z-1)(1-yz)}, & \text{(green) }\textcolor{green!80!black}{f_{g}}&:=(xyz)^{2},\\
\text{(orange) } \textcolor{orange}{f_{o}}&:=\frac{1+x}{(z-1)(yz-1)},  &\text{(cyan) }\textcolor{cyan}{f_{c}}&:= \frac{1+x}{(1-z)(yz-1)},  & \text{(violet) }\textcolor{violet}{f_{v}}&:=x
\end{align*}
where $x:=\hilbt^{\left[\begin{smallmatrix} 1 \\ 0 \\ 0\end{smallmatrix}\right]}$, $y:=\hilbt^{\left[\begin{smallmatrix} 0 \\ 1 \\ 0\end{smallmatrix}\right]}$, and $z:=\hilbt^{\left[\begin{smallmatrix} 0 \\ 0 \\ 1\end{smallmatrix}\right]}$. We may classify $\rchaincpx{\crosssection_{G}}$ as follows:
\begin{align*}
\rchaincpx{\crosssection_{\textcolor{red}{r}}}:  & 0 \to \field \to \field^2 \to \field \to 0& \rchaincpx{\crosssection_{\textcolor{blue}{b}}},\rchaincpx{\crosssection_{\textcolor{cyan}{c}}}: &0 \to 0 \to \field \to \field \to 0 \\ 
\rchaincpx{\crosssection_{\textcolor{green!80!black}{g}}}: & 0 \to \field \to 0 \to 0 \to 0 & \rchaincpx{\crosssection_{\textcolor{violet}{v}}}: &0 \to 0 \to \field^{2}\to \field \to 0  \\
\rchaincpx{\crosssection_{\textcolor{orange}{o}}}: & 0 \to 0 \to 0\to \field \to 0
\end{align*}
Hence, 
\begin{align*}
\hilb{\localcoho{\maxid}{0}{\setS}}{\{x,y\}} &=\textcolor{green!80!black}{f_{g}} & \hilb{\localcoho{\maxid}{1}{\setS}}{\{x,y\}}&=\textcolor{violet}{f_{v}} & \hilb{\localcoho{\maxid}{2}{\setS}}{\{x,y\}}&=\textcolor{orange}{f_{o}}.
\end{align*}
\end{enumerate}
\end{example}

\section{Revisiting Cohen--Macaulay affine semigroup rings}
\label{sec:Cohen--Macaulay}

In this section we concentrate on the special case when $\midI=0$. We
give a new criterion using grains to detect whether
$\field[\affinesemigp]$ is Cohen--Macaulay and give an alternative
proof of the Cohen--Macaulay condition in~\cite{MR857437}. To begin,
we recall a celebrated theorem of Hochster~\cite{MR304376} when
$\affinesemigp$ is normal, and prove it yet again with our methods.

\begin{theorem}[\cite{MR304376}]
\label{thm:case_of_normal_affine_semigroup}
If $\affinesemigp$ is normal, $\field[\affinesemigp]$ is Cohen--Macaulay. 
\end{theorem}

To prove this, we recall concepts of~\cite{BH_CMrings}. Given a
polyhedron $\polyhedron$ in $\R$-vector space $V$, let $\veca, \vecb
\in V$ two distinct points. If $[\veca,\vecb]$ does not contain a
point $\vecb' \in \polyhedron$ with $\vecb' \neq \vecb$, we say that
$\vecb$ is \emph{visible} from $\veca$. A subset $\setS$ is
\emph{visible} if every $\vecb \in \setS$ is visible. Given a polytope
$\polytope'$, a contractible polyhedral subcomplex is formed by the
set of all visible points from $\veca \in V\smallsetminus \polytope'$~\cite{BH_CMrings}*{Proposition 6.3.1}. We refer to this polyhedral subcomplex as the \emph{$\veca$-visible subpolytope} of $\polyhedron'$.

\begin{proof}
Let $\arngmt$ be the hyperplane arrangement generated by hyperplanes in the $H$-representation of $\realize{\affinesemigp}$. We claim that 
\begin{equation}
  \label{eqn:visible}
\grainset(\field[\affinesemigp]) = \{\regr_{\setS} \cap
\Z\affinesemigp\mid \regr_{\setS} \in \regionr{\arngmt}\}.
\end{equation}
If the equality~\eqref{eqn:visible} holds, for the given nonempty
$\setS$ and a point $\veca \in \regr_{\setS} \cap \Z\affinesemigp$,
construct a hyperplane $\hplane$ containing $\veca$ and transversally
intersecting $\realize\affinesemigp$. The transverse section
$\crosssection$ is then realized as a poytope $\hplane \cap
\realize\affinesemigp$. Thus, the chaff of a grain can be defined as a
subset of $\facelattice(\crosssection)$ that is not visible from
$\veca$. Due to the contractibility of both $\crosssection$ and
$\veca$-visible subpolytope of $\crosssection$, the chain complex over
the chaff is contractible via the long exact sequence of
cohomology. Thus, except the top dimension, the chaff of any grain has
vanishing homology. This argument essentially paraphrases the proof
of~\cite{BH_CMrings}*{Theorem 6.3.4}. 

To prove~\eqref{eqn:visible}
recall that the poset of regions $\regionr{\arngmt}$ partitions
$\degs(\field[\affinesemigp]) = \Z\affinesemigp$. We may use induction
over the cardinality of $\setS$ to determine that each grain is of the
form $\regr_{\setS} \cap \Z\affinesemigp$. Assume that the hyperplane
arrangement $\arngmt$ consists of the elements
$\hplane_{1},\hplane_{2},\cdots, \hplane_{m}$, and that
$\facF_{\smalli}$  is the facet supported by $\hplane_{\smalli}$ for $i \leq m$. Start with $|\setS|=m$; $\regr_{\setS} \cap \Z\affinesemigp$ is a grain since $(0,\affinesemigp)$ is the unique degree pair of $\affinesemigp$ and $0+\affinesemigp \subset 0+\left(\affinesemigp-\N\facF\right)$ for any face $\facF$.

To use induction, suppose we showed that $\regr_{\setS} \cap \Z\affinesemigp$ with $|\setS| \leq m-\smalli$ is a grain. Then, for any $\setS$ with cardinality $m-(\smalli+1)$, we claim $\regr_{\setS} \cap \Z\affinesemigp = (\bigcap_{\smalli \in \setS}(\affinesemigp-\N\facF_{\smalli})) \minus \left( \bigcup_{\setT \supsetneq \setS}\regr_{\setT} \right) \cap \Z\affinesemigp$. Indeed, $(\bigcap_{\smalli \in \setS}(\affinesemigp-\N\facF_{\smalli})) = \regR_{\setS} \cap \Z\affinesemigp$ by the definition of the cumulative regions and the normality of $\affinesemigp$. Then, the righthand side $(\bigcap_{\smalli \in \setS}(\affinesemigp-\N\facF_{\smalli})) \minus \left( \bigcup_{\setT \supsetneq \setS}\regr_{\setT} \right) \cap \Z\affinesemigp$ is nothing more than the construction of $\regr_{\setS}$ from $\regR_{\setS}$. Furthermore, for each $\setT \supsetneq \setS$,$\regr_{\setT} \cap \Z\affinesemigp$ is a grain by inductive hypothesis. This shows the proposed one-to-one correspondence between grains and regions in $\regionr{\arngmt}$.

Now pick a grain $\regr_{\setS} \cap \Z\affinesemigp$. Then $\regr_{\setS} \subseteq \regR_{\setT} \cap \Z\affinesemigp$ if and only if $\setT \subseteq \setS$. Hence, its chaff can be identified as a subset of faces in $\facelattice(\affinesemigp)$ whose corresponding localizations contain $\regr_{\setS} \cap \Z\affinesemigp$.
\end{proof}

When $\affinesemigp$ is not normal, we need the chaffs and grains of
the module $\field[\affinesemigp_{\text{sat}}]/\field[\affinesemigp]$
to determine the chaffs of grains of $\field[\affinesemigp]$. To
distinguish two chaffs and grains from different modules, we refer to
the grains and chaffs of the module
$\field[\affinesemigp_{\text{sat}}]/\field[\affinesemigp]$ as
\emph{void grains and void chaffs}, respectively, in accordance with
the conventions in~\cref{ex:std_void}(\ref{enum:voidpairs}). 

\begin{theorem}
\label{thm:case_of_non_normal_affine_semigroup}
$\field[\affinesemigp]$ is Cohen--Macaulay if and only if every grain
consisting of void grains has vanishing homology except in top
dimension. 
\end{theorem}

\begin{proof}
If a grain $\grain$ has a degree which exists in $\bigcup\deg(\affinesemigp_{\text{sat}})$, then we may apply the same argument of~\cref{thm:case_of_normal_affine_semigroup} to show that the homology of $D_{\grain}$ vanishes except for the top dimension. Hence, the only grains we need to investigate the homology of their chaff are a grain consisting of holes.
\end{proof}

Now we are prepared to give an alternative proof of the main result
of~\cite{MR857437}. Let $\facF_{1}, \cdots, \facF_{m}$ be facets of a
pointed affine semigroup $\affinesemigp$. Let
$\widetilde{\affinesemigp}:=
\bigcap_{i=1}^{m}(\affinesemigp-\N\facF_{i})$. For any nonempty subset
$\setS$ of $\{ 1,2,\cdots, m\}$, let $\grain_{\setS}:=
\bigcap_{i\not\in \setS}(\affinesemigp-\N\facF_{i})  \minus
\bigcup_{j\in \setS}(\affinesemigp-\N\facF_{j})$. Let $\pi_{\setS}$ be
the simplicial complex of nonempty subsets $I$ of $\setS$ such that
$\bigcap_{i \in I}F_{i}$ is a nonempty face of $\affinesemigp$. By
abuse of notation, we identify the face lattice
$\facelattice(\pi_{\setS})$ of $\pi_{\setS}$ as a subset $\{\bigcap_{i
  \in I}F_{i} \in \facelattice(\affinesemigp)\smallsetminus \{
\varnothing\} \mid I \in \pi_{\setS}  \}\cup\{\varnothing\}$ of
$\facelattice(\affinesemigp)$. We say $\pi_{\setS}$ is \emph{acyclic}
if its reduced homology group is zero for all indices. 

\begin{theorem}[Main theorem in~\cite{MR857437}]
\label{thm:TH_MainThm}
$\field[\affinesemigp]$ is a Cohen--Macaulay ring if and only if (1) $\widetilde{\affinesemigp} = \affinesemigp$ and (2) for every $\setS \subseteq \{ 1,2, \cdots, m\}$ with $\regR_{\setS}\in \regionr{\arngmt}$, $\pi_{\setS}$ is acyclic.
\end{theorem}
\begin{proof}
Note that for any $\veca \in \degs(\field[\affinesemigp])
\smallsetminus \bigcup_{i=1}^{m}\deg\left(\affinesemigp
  -\N\facF_{i}\right)$, $\crosssection_{\veca}=\{Q\}$, which therefore
only contributes to the $(\dim Q)$-th local cohomology. Thus, to prove the conditions above imply Cohen--Macaulayness, it suffices to show that for any $\veca \in \bigcup_{i=1}^{m}\affinesemigp -\N\facF_{i}$, $\rchaincpx{\crosssection_{\veca}}$ is exact. Since $G_{\setS}$ partitions $\bigcup\deg(\field[\affinesemigp])$, assume $\veca \in G_{\setS}$ for some proper subset of $\{ 1,2,\cdots, m\}$. Then, for any $\facF \in \facelattice(\pi_{\setS})^{c}$, $\facF \not\subset \facF_{i}$ for any $i \in \setS$, since $\facF_{i} \in \pi_{\setS}$. Thus, $\facF = \bigcap_{i \in J}F_{i}$ for some $J\subset \{1,2,\cdots, m\}\minus \setS$ implies that $\affinesemigp-\N\facF$ contains $\veca$. Conversely, for any face $\facG$ of $\bigcap_{i \in \setS}\facF_{i}$, $\veca \not\in \affinesemigp-\N\facG$. Hence $\crosssection_{\veca}=\facelattice(\pi_{\setS})^{c}$ by identifying $\facelattice(\pi_{\setS})$ as a subset of $\facelattice(\affinesemigp).$ In this identification, $\pi_{\setS}$ is isomorphic to a polyhedral subcomplex of the transverse section $\crosssection$ of $\realize\affinesemigp$. Hence, the complements $\facelattice(\pi_{\setS})^{c}$ form a polyhedral subcomplex of the dual polytope $\crosssection^{\text{dual}}$. Apply Alexander duality to conclude that $\crosssection_{\veca}$ is acyclic. This proves that $\field[\affinesemigp]$ is Cohen--Macaulay in accordance with~\cref{thm:Cohen--Macaulayness}.

Conversely, suppose $\pi_{S}$ is not acyclic. Pick $\veca \in
\grain_{\setS}$ such that $\crosssection_{\veca}=
\facelattice(\pi_{\setS})^{c}$. Now Alexander duality ensures that
$\rchaincpx{\facelattice(\pi_{\setS})^{c}}$ has nontrivial cohomology
at $i$-th index which is less than $(\dim \Q)$. Also, if $Q' \neq
Q$,~\cref{cor:for_reproving_trung_hoa} gives a void pair
$(\vecb,\facF)$ with $\dim\facF\leq \dim Q-2$. Let $\setS$ be a set of
indices of hyperplanes containing
$\facF$.By~\cref{pro:localization_cone_structure} there exists $\veca
\in \left(\vecb+\N(\facF\cup(-\facF))\right) \cap
\regr_{\setS}$. Hence, $\crosssection_{\veca}=(Q/\facF) \smallsetminus
\facF := \{\facG \in \facelattice(\affinesemigp)\mid Q \supseteq \facG
\supsetneq \facF \}$ is combinatorially equivalent to the polytope
$(Q/\facF)^{\operatorname{dual}}$ without its relative
interior. Hence, $(\dim\facF)$-th homology of $\crosssection_{\veca}$
is nonzero, as is the $(\dim Q-\dim\facF)$-th local cohomology. Thus,
the $\veca$-graded part of the Ishida complex admits nonzero local
cohomology with an index less than $\dim Q$, indicating that the
semigroup ring is not Cohen--Macaulay.
\end{proof}

\begin{example}[Continuation of \cref{ex:hochster}]
\label{ex:cmness}
\
Both cases are not Cohen--Macaulay due to the presence of nonzero 0-th local cohomology.
\end{example}
\begin{example}[Examples of non-normal affine semigroup rings]
\label{ex:nonnormal_cmness}
\
\begin{enumerate}[leftmargin=*]
\item (A 3-dimensional non-Cohen--Macaulay affine semigroup ring)
\label{enum:3d_nonnormal_ex_degp_top}
Let $\affinesemigp:=\N\left[\begin{smallmatrix} 1& 1 & 1 & 1 & 1 \\ 0
    & 0 & 0 & 1 & 1 \\ 0 & 1 & -2 & 0 & 1 \end{smallmatrix}\right].$
Index the rays, facets, and hyperplanes as follows
\[
\resizebox{\linewidth}{!}{
\begin{tabular}{cccc}
$\langle\veca_{1}\rangle := \left\langle\left[\begin{smallmatrix} 1 \\ 0 \\ 1\end{smallmatrix}\right]\right\rangle $ & 
$\langle\veca_{2}\rangle:= \left\langle\left[\begin{smallmatrix} 1 \\ 0 \\ -2\end{smallmatrix}\right]\right\rangle $ &
$\langle\veca_{3}\rangle:= \left\langle\left[\begin{smallmatrix} 1 \\ 1 \\ 0\end{smallmatrix}\right]\right\rangle$ & 
$\langle\veca_{4}\rangle:= \left\langle\left[\begin{smallmatrix} 1 \\ 1 \\ 1\end{smallmatrix}\right]\right\rangle$ \\
$\facF_{1}:= \langle \veca_{1},\veca_{2} \rangle$ &
$\facF_{2}:= \langle \veca_{2},\veca_{3} \rangle$ &
$\facF_{3}:= \langle \veca_{3},\veca_{4} \rangle$ & 
$\facF_{4}:= \langle \veca_{4},\veca_{1} \rangle $\\
$\hplane_{1} := \{y=0 \}$ & 
$\hplane_{2}:= \{ 2x-2y+z=0 \} $ &
$\hplane_{3} := \{x-y=0\} $ & 
$\hplane_{4} := \{x-z=0\}$
\end{tabular}}
\]
The Hasse diagram for the region of posets is identical to that
in~\cref{fig:3d_hasse_diagram}. Moreover, the set of holes of the
affine semigroup, $\holes(\affinesemigp)$, is
$\left[\begin{smallmatrix} 1 \\ 0 \\
    -1 \end{smallmatrix}\right]+\N\left[\begin{smallmatrix} 1 \\ 0 \\
    -2 \end{smallmatrix}\right]$ which lies in the $xz$-plane. This is
because $\left[\begin{smallmatrix} 2 \\ 1 \\
    -1 \end{smallmatrix}\right] = \left[\begin{smallmatrix} 1 \\ 1 \\
    1 \end{smallmatrix}\right]+ \left[\begin{smallmatrix} 1 \\ 0 \\
    -2 \end{smallmatrix}\right]$ acts as a barrier to the spread of
holes in the relative interior of $\affinesemigp$. According
to~\cref{thm:map_between_sets_of_equiv_classes_of_degree_pairs},
$\affinesemigp$ and $\affinesemigp -\N\veca_{2}$ are the only
non-normal affine semigroups that arise as a result of
localization. Thus, the space of holes $\voidm(\affinesemigp)$ equals
$ \left[\begin{smallmatrix} 1 \\ 1 \\ 1 \end{smallmatrix}\right]+
\Z\left[\begin{smallmatrix} 1 \\ 0 \\ -2 \end{smallmatrix}\right]$,
which is consistent with the set of holes of
$\affinesemigp-\N\veca_{2}$. Using~\cref{fig:3d_hasse_diagram},
$\voidm(\affinesemigp)$ is decomposed into void grains below. 
\[
\begin{tabular}{|c|c|c|}
\toprule
$\holes_{1,2,3,4}$ & $\holes_{1,2,3}$ & $\holes_{1,2}$ \\
\midrule
$\left[\begin{smallmatrix} 1 \\ 1 \\ 1 \end{smallmatrix}\right]+ \N\left[\begin{smallmatrix} 1 \\ 0 \\ -2 \end{smallmatrix}\right]$ & $\left[\begin{smallmatrix} 0 \\ 1 \\ -1 \end{smallmatrix}\right]$ & $\left[\begin{smallmatrix} 1 \\ 1 \\ 1 \end{smallmatrix}\right]+ \Z_{\leq -2}\left[\begin{smallmatrix} 1 \\ 0 \\ -2 \end{smallmatrix}\right]$ \\
\bottomrule
\end{tabular}
\]
$\holes_{1,2,3,4}= \bigcap_{\smalli=1,3,4}(\affinesemigp-\N\veca_{\smalli}) \smallsetminus \affinesemigp$ and $\holes_{1,2} = (\affinesemigp-\N\facF_{1}) \cap( \affinesemigp-\N\facF_{2}) \smallsetminus (\affinesemigp-\N\veca_{2})$, $\holes_{1,2,3,4}$ and $\holes_{1,2}$ form grains. On the other hand,
\[
\holes_{1,2,3} \subsetneq \grain:=(\affinesemigp-\N\veca_{3}) \smallsetminus \left( \affinesemigp \cup \left(\affinesemigp-\N\veca_{2} \right)\cup\holes_{1,2,3,4}\right)
\]
shows that $\holes_{1,2,3}$ is a part of $\grain$, whereas the remaining elements of $\grain$ come from the region $\regr_{2,3}$. Thus,
\begin{align*}
\grainset(\affinesemigp/\idI) &= \{\holes_{1,2,3,4}, \holes_{1,2}, \affinesemigp, (\affinesemigp-\N\veca_{2})\cap\regr_{1,2}, \left(\regr_{2,3}\cap\left(\affinesemigp-\N\veca_{3}\right)\right)\cup\holes_{1,2,3} \} \\
&  \cup \{\Z^{3}\cap \regr_{\setS} \mid \setS \in \text{index}(\regionr{\arngmt})  \text{ such that }\setS \neq \{1,2\}, \{2,3\}, \{1,2,3,4\} \},
\end{align*}
where $\text{index}(\regionr{\arngmt})$ denotes the set of all indices of elements of $\regionr{\arngmt}$. Hence, it suffices to check whether chaffs 
\[
D_{\holes_{1,2,3,4}}=\{ \veca_{i}, \facF_{j},\affinesemigp
\}_{\substack{i=1,3,4 \\ j=1,2,3,4}} \text{ and }D_{\holes_{1,2}}=\{
\facF_{1},\facF_{2},\affinesemigp \}
\]
have vanishing homology. Since $D_{\holes_{1,2}}$ produces a non-zero second homology, the affine semigroup ring $\field[\affinesemigp]$ is not Cohen--Macaulay.
\item (4-dimensional non-normal Cohen--Macaulay affine semigroup
  ring~\cite{MR2508056}*{Exercise 6.4}) Assume $P$ is a simplex with
  the vertices $(0,0,0),(2,0,0),(0,3,0),$ and $(0,0,5)$. $\Z^{3}$ is
  the smallest lattice that contains vertices of $P$. The
  \emph{polytopal affine monoid} $M(P)$~\cite{MR2508056}*{Definition
    2.18} associated $P$ is the affine semigroup
  $\affinesemigp:=\N\genset$ where $\genset =\{(1,\veca): \veca \in
  \Z^{3}\cap P\}$. In this example,
\[
  \genset =\left[ \begin{smallmatrix}\veca_{1}& \veca_{2}&\cdots &
      \veca_{18} \end{smallmatrix}\right]
  =\left[\begin{smallmatrix}1&1&1&1&1&1&1&1&1&1&1&1&1&1&1&1&1&1 \\
      0&0&0&0&0&0&0&0&0&0&0&0&0&1&1&1&1&2 \\
      0&0&0&0&0&0&1&1&1&1&2&2&3&0&0&0&1&0 \\
      0&1&2&3&4&5&0&1&2&3&0&1&0&0&1&2&0&0 \end{smallmatrix}\right].
\]
$\realize\affinesemigp$ is a simplicial polyhedron with tetrahedral
transverse section. We index its facets and hyperplanes as follows
\[
\begin{tabular}{cccc}
$\facF_{1}:= \langle \veca_{6},\veca_{13}, \veca_{18} \rangle$ &
$\facF_{2}:= \langle \veca_{1},\veca_{7}, \veca_{11},\veca_{13},\veca_{14},\veca_{17},\veca_{18} \rangle$ \\
$\facF_{3}:= \langle \veca_{1}, \cdots, \veca_{6}, \veca_{14},\cdots ,\veca_{16},\veca_{18} \rangle$ & 
$\facF_{4}:= \langle \veca_{1},\cdots,\veca_{13} \rangle $\\
$\hplane_{1} := \{(30,-15,-10,-6) \cdot \vect = 0 \}$ & 
$\hplane_{2}:= \{ (0,0,0,1) \cdot \vect = 0 \} $ \\
$\hplane_{3} := \{(0,0,1,0) \cdot \vect = 0\} $ &
$\hplane_{4} := \{(0,1,0,0) \cdot \vect = 0\}$
\end{tabular}
\]
Since the transverse section $\crosssection$ is a tetrahedron, we can
index faces as intersections of facets uniquely. For example, each of the rays can be denoted as follows.
\begin{align*}
\facF_{2,3,4} &:= \left\langle\veca_{1}\right\rangle  & \facF_{1,3,4}& := \left\langle\veca_{6}\right\rangle   & \facF_{1,2,4} &:= \left\langle\veca_{13}\right\rangle  & \facF_{1,2,3}& := \left\langle\veca_{18}\right\rangle.
\end{align*}
According to the HASE package~\cite{KLRY18}, the affine semigroup $\affinesemigp=\N\genset$ contains holes $\holes(\affinesemigp):=\left[\begin{smallmatrix} 2 & 1 & 2& 4 \end{smallmatrix}\right]^{t}+\N\facF_{1}$. Thus, the space of holes $\voidm(\affinesemigp)$ is $\left[\begin{smallmatrix} 2 & 1 & 2& 4 \end{smallmatrix}\right]^{t}+\Z\facF_{1}$. We can decompose $\voidm(\affinesemigp)$ into void grains using the hyperplane arrangement as follows:
\[
\holes_{\setS}:=\left\{\left[\begin{smallmatrix} 2+a+b+c \\ 1+2a \\ 2+3b \\ 4+5c \end{smallmatrix}\right]\mid a \in\text{sgn}_{4}(S), b \in \text{sgn}_{3}(S), c \in \text{sgn}_{2}(S)\right\} 
\]
where $\text{sgn}_{i}(S):= \begin{cases} \N & \text{ if } i \in \setS \\ \Z \smallsetminus \N & \text{ if } i \not\in \setS \end{cases}$ for all $\{1\} \subseteq \setS \subseteq \{ 1,2,3,4\}$. There are two types of grains indexed by $2^{\{1,2,3,4\}}$ that emerge from iterative intersections of affine semigroups. For every $\setS$ with $\{1\} \subseteq \setS \subseteq \{ 1,2,3,4\}$, the union $\grain_{\setS \smallsetminus \{1\}}:=\holes_{\setS} \cup\left(\regr_{\setS \smallsetminus \{ 1\}} \cap \left(\affinesemigp - \N\facF_{\setS\smallsetminus \{1\}}\right)\right)$ generates a grain of the first type, whereas $\grain_{\setS}:=\regr_{\setS}\cap\left(\affinesemigp-\N\facF_{\setS}\right)$ generates a grain of the second type. Since there is no grain composed entirely of holes, all chaffs have vanishing homology except the top dimension. Therefore, $\field[\affinesemigp]$ is Cohen--Macaulay.
\end{enumerate}
\end{example}

\bibliographystyle{plain}
\bibliography{CMness}
\end{document}